\documentclass{aims_arxiv}

\usepackage{txfonts}

\usepackage{amsmath, amsfonts, amssymb, amsthm}
\usepackage[utf8]{inputenc}
\usepackage{tikz}
\usepackage{tikz-cd}
\usepackage{algorithm}
\usepackage{algpseudocode}
\usetikzlibrary{patterns, quotes, angles, hobby}
\usepackage{subcaption}
\tikzset{
	dot/.style = {circle, fill, minimum size=#1,
		inner sep=0pt, outer sep=0pt},
	dot/.default = 6pt 
}
\usepackage{xcolor}
\newtheorem{theorem}{Theorem}
\newtheorem{example}{Example}
\newtheorem{remark}{Remark}
\newtheorem{definition}{Definition}
\newtheorem{assume}{Assumption}
\newtheorem{axiom}{Axiom}
\newtheorem{problem}{Problem}
\newtheorem{proposition}{Proposition}
\newtheorem{corollary}{Corollary}
\definecolor{myorange}{RGB}{255, 180, 30}

\def\typeofarticle{}
\def\currentvolume{17}
\def\currentissue{x}
\def\currentyear{2025}

\def\ppages{xxx--xxx}
\def\DOI{}
\def\Received{22 April 2025 }
\def\Revised{17 September 2025}
\def\Accepted{24 September 2025}
\def\Published{xx September 2025}
\usepackage{cite}
\usepackage{microtype}

\numberwithin{equation}{section}

\begin{document}

\title{Symplectic geometry in hybrid and impulsive optimal control}

\author{%
	William Clark\affil{1,}\corrauth and
	Maria Oprea\affil{2}
}

\shortauthors{the Author(s)}

\address{%
	\addr{\affilnum{1}}{Department of Mathematics, Ohio University, Athens, OH 45701, USA}
	\addr{\affilnum{2}}{Center for Applied Mathematics, Cornell University, Ithaca, NY 14853, USA}
}
\corraddr{clarkw3@ohio.edu.
}

\begin{abstract}
	Hybrid dynamical systems are systems which undergo both continuous and discrete transitions. The Bolza problem from optimal control theory was applied to these systems and a hybrid version  of Pontryagin's maximum principle was presented. This hybrid maximum principle was presented to emphasize its geometric nature which made its study amenable to the tools of geometric mechanics and symplectic geometry. One explicit benefit of this geometric approach was that the symplectic structure (and hence the induced volume) was preserved. This allowed for a hybrid analog of caustics and conjugate points. Additionally, an introductory analysis of singular solutions (beating and Zeno) was discussed geometrically. This work concluded on a biological example where beating can occur.
\end{abstract}

\keywords{hybrid systems; optimal control; symplectic geometry
	\newline
	\textbf{Mathematics Subject Classification:} 49N25, 34A38, 37J39}

\maketitle

\section{Introduction}

Solving an optimal control problem amounts to finding the ``best'' path that connects an initial location $A$ to a terminal location $B$. Analytically this is described by finding the minimizer of a functional, $J[\gamma]$, where $\gamma$ is a continuous path connecting the points $A$ and $B$. An alternative approach to this problem is to consider the optimal path geometrically; the best path can be interpreted as the ``straight path'' connecting the two points. This interpretation is inspired by classical mechanics where trajectories follow the principle of least action (i.e., best paths are minimizers of the action functional) and the optimal path follows the equations of motion described by the famous Euler-Lagrange equations (which are precisely geodesic equations for natural systems with no potential energy). The goal of this work is to offer an in-depth analysis of this analogy for hybrid dynamical systems - systems whose trajectories are subject to both continuous-time and discrete-time update laws. More particularly, we are searching for the minimizer of $J[\gamma]$ where $\gamma$ is now allowed to undergo discrete transitions, and characterizing the optimal discontinuous path as a  ``straight lines'' in the hybrid framework. Hybrid dynamical systems are used to describe a wide variety of disciplines from robotics, to biological and medical systems \cite{review_hybrid}. The problem of optimal control of hybrid dynamical systems has been studied extensively in recent years. One line of study aims to express a hybrid system  in a form that allows for the application of classical optimal control techniques such as the Pontryagin maximum principle (PMP) \cite{azhmyakovNec} and dynamic programming \cite{pakniyat2014minimum} from the continuous setting. For instance, \cite{westenbroekcdc} smooths out the impact surface using tools perturbation theory and analyzes the hybrid control problem as a limiting case of the obtained continuous one. Alternatively, \cite{yunt_mechanical_impact} uses measure differential inclusions to treat the hybrid dynamical system as an ordinary differential equation with a set valued map, and \cite{azhmyakovMax} interprets trajectories as measures, with impacts corresponding to Dirac deltas at the contact point. Solutions for the particular case where the cost is simply the time to origin, and the number of impacts is fixed have been proposed in \cite{cristofarocdc}. The other line of study tackles the hybrid system directly by extending the maximum principle and dynamic programming to account for discontinuities \cite{pakniyat2014minimum,pakniyat2015minimum, pekarek2008variational, liberzon_oc}. Our goal is to analyze this setting from a geometric point of view. More formally, the general problem this work considers will be to find solutions to the following minimization problem:
\begin{equation*}
	\begin{gathered}
		u^* = \arg\min_{u(\cdot)}\, J[x_0, u(\cdot)], \\ 
		J[x_0, u(\cdot)] := \int_0^{t_f}\, \ell\left(x(\tau),u(\tau)\right)\, d\tau + g\left( x(t_f)\right),
	\end{gathered}
\end{equation*}
where $x$ is subject to the controlled hybrid dynamics:
\begin{equation}\label{eq:controlled_HDS}
	\begin{cases}
		\dot{x} = f(x,u), & x\not\in S, \\
		x^+ = \Delta(x^-), & x\in S,
	\end{cases}
\end{equation}
where $x\in M$ is the ambient space and $S\subset M$ is the set where the discrete transitions occur. 

The hybrid maximum principle \cite{pakniyat2015minimum} states that the usual maximum principle applies during continuous transitions while a ``Hamiltonian jump condition'' will apply at the moment of a discrete transition. As an elementary example of this procedure, suppose we wish to find the shortest path connecting two points in $\mathbb{R}^2$, $A = (x_1,y_1)$ and $B=(x_2,y_2)$ with $y_i<0$. The shortest path is clearly the straight line connecting the two points. Suppose now that the path is \textit{required} to touch the line $\sigma = \{y=0\}$ and does so at the undetermined point $C = (z,0)$ as depicted in Figure \ref{fig:intro_bounce}. This constrained shortest path will be the concatenation of the straight line connecting points $A$ and $C$ with the line connecting points $C$ and $B$. Its length is (as a function of $z$):
\begin{equation*}
	D(z) = \sqrt{ (x_1-z)^2 + y_1^2} + \sqrt{(x_2-z)^2+y_2^2}.
\end{equation*}
The value of $z$ that minimizes this function is
\begin{equation*}
	z^* = \frac{x_1y_2+x_2y_1}{y_1+y_2}.
\end{equation*}
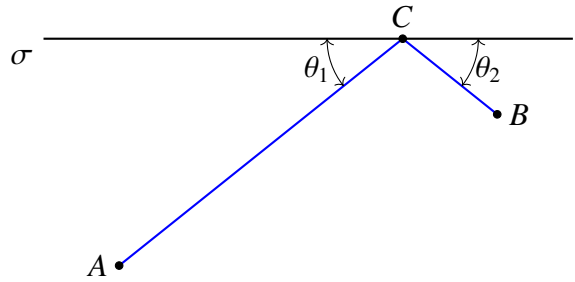
\begin{figure}
	\centering
	\begin{tikzpicture}
		\draw[black,thick] (-3,0) -- (4,0);
		\node[below left] at (-3,0) {$\sigma$};
		\draw[blue,thick] (-2,-3) -- (1.75,0);
		\draw[blue,thick] (1.75,0) -- (3,-1);
		\node[left] at (-2,-3) {$A$};
		\node[above] at (1.75,0) {$C$};
		\node[right] at (3,-1) {$B$};
		\draw[fill] (-2,-3) circle [radius=0.05];
		\draw[fill] (1.75,0) circle [radius=0.05];
		\draw[fill] (3,-1) circle [radius=0.05];
		\path  (-2,-3) coordinate (a)
		-- (1.75,0) coordinate (b)
		-- (-3,0) coordinate (c)
		pic["$\theta_1$", draw=black, <->, angle eccentricity=1.2, angle radius=1cm] {angle=c--b--a};
		\path  (3,-1) coordinate (A)
		-- (1.75,0) coordinate (B)
		-- (4,0) coordinate (C)
		pic["$\theta_2$", draw=black, <->, angle eccentricity=1.2, angle radius=1cm] {angle=A--B--C};
	\end{tikzpicture}
	\caption{The shortest path connecting points $A$ and $B$ while also touching the line $\sigma$.}
	\label{fig:intro_bounce}
\end{figure}
It is important to note that the two angles, $\theta_1$ and $\theta_2$, are equal. Indeed,
\begin{equation*}
	\tan\theta_1 = \frac{-y_1}{z^*-x_1} = \frac{y_1+y_2}{x_1-x_2} = \frac{-y_2}{x_2-z^*} = \tan\theta_2.
\end{equation*}
This means that the shortest path has its angle of incidence equal to its angle of reflection, i.e., the shortest path obeys specular reflection. 

The general hybrid maximum principle is a generalization of this example. Between points $A$ and $C$ and between $C$ and $B$, the optimal path is a straight line which corresponds to a Hamiltonian trajectory. At the point $C$, the path undergoes specular reflection, which is the Hamiltonian jump condition.

The contribution of this work is to provide a geometric interpretation of solutions to the optimal control problem of hybrid systems through the interpretation of the PMP \cite{pontryagin} in a differential geometric/symplectic viewpoint \cite{linan_geometric}. First, we give a geometric proof of the Hybrid PMP (HPMP) through the augmented differential (see Definition \ref{def:aug_dif})  and derive sufficient conditions for existence and uniqueness of solutions. Second, we show that the Hamiltonian jump condition is a symplectomorphism and, as a consequence, the hybrid evolution of the adjoint equation is symplectic (and, hence, volume-preserving). Third, we extend the ideas of \cite{lagrangian_manifolds} to the hybrid case so as to give a global characterization of solutions as intersections of Lagrangian submanifolds. 
This is particularly useful since it allows us to classify points where the resulting boundary value problem becomes singular as points where the Lagrangian submanifolds undergo folding.

Lastly, we examine cases in which the HPMP described in Theorem \ref{thm:HPMP} does not apply. As opposed to \cite{pakniyat2015minimum} where the cases of codimension 2 or higher impacts are considered, here we turn to the the case in which the hybrid trajectory is not regular, i.e., the trajectory is beating (two or more impacts may happen simultaneously) or Zeno (infinitely many impacts in a finite amount of time). For the beating case, we extend the corner conditions to determine the co-state jumps on the beating sets.  Although a result from \cite{clark2023invariant} states that Zeno trajectories almost-never occur in hybrid systems which possess an invariant volume-form, we show that the optimal trajectory might indeed benefit from the Zeno effect. 

This paper is organized as follows: the definition for both hybrid and impact systems along with their solutions are presented in Section \ref{sec:hybrid}. The optimal control problem is described in Section \ref{sec:problem}, and necessary conditions for optimality are presented in Section \ref{sec:necessary}. Theorem \ref{thm:HPMP} is the main theorem of this section and it presents the HPMP for regular arcs. The definition of  hybrid Lagrangian submanifolds  is presented in Section \ref{sec:hybrid_Lagrange}, and Subsections \ref{subsec:intersect} and \ref{subsec:caustics} analyze intersections and caustics of hybrid Lagrangian submanifolds. Section \ref{sec:Zeno} is split into two parts: Subsection \ref{subsec:beating} studies the beating phenomena, and Subsection \ref{subsec:Zeno} is dedicated to Zeno executions. A biological application with numerical simulations is presented in Section \ref{sec:example}. This paper concludes with Section \ref{sec:conclusions}.

Summation notation will be used throughout; an upper and lower index implies a sum, $a_ib^i=\sum\,a_ib^i$. Additionally, throughout this work, most of the objects considered will be smooth, which is taken to mean $C^\infty$. 
\section{Hybrid and impact dynamical systems}
\label{sec:hybrid}
Hybrid dynamical systems are used to study and model a wide range of phenomenon. As such, a  variety of different definitions exist in the literature. This section introduces our definition of a hybrid system as well as the important special case of impact systems. 
\subsection{Hybrid systems}
As there exists a plethora of definitions for a hybrid dynamical system, the specific class discussed in this paper will follow the definition below.
\begin{definition}[Hybrid Dynamical Systems]\label{def:HDS}
	A hybrid dynamical system, abbreviated HDS, is a 4-tuple $\mathcal{H} = (M,S,X,\Delta)$ where
	\begin{enumerate}
		\item[1.] $M$ is a smooth (finite-dimensional, $\dim(M)=n$) manifold,
		\item[2.] $S\subset M$ is a smooth embedded submanifold with co-dimension 1, called the \textit{guard},
		\item[3.] $X:M\to TM$ is a smooth vector field, and
		\item[4.] $\Delta:S\to M$ is a smooth map, called the reset, such that $\Delta(S)$ is also a smooth embedded submanifold and there exists a continuous extension to its closure, $\overline{\Delta}:\overline{{S}}\to M$.
	\end{enumerate}
	The manifold $M$ will be referred to as the state-space, $S$ as the guard, and $\Delta$ as the reset. Moreover, a regular HDS (rHDS) additionally satisfies:
	\begin{enumerate}
		\item[5.] $S\cap\Delta(S)=\emptyset$ and $\overline{S}\cap\overline{\Delta(S)}\subset M$ have co-dimension at least 2.
	\end{enumerate}
\end{definition}
The importance of $\Delta(S)$ being a submanifold is that it allows for a more graceful study of differential forms across resets. 
The guard having codimenion 1 guarantees that resets are not too rare, as shown in Figure \ref{fig:codim2}, and are needed for the quasi-smooth dependence property; see Definition \ref{def:qsdp} below. Condition (5) is a regularity assumption and is needed to prove Theorem \ref{thm:Zeno} below. The set $S$ as defined above does not depend on time - this can be overcome by appending time as an extra dimension.
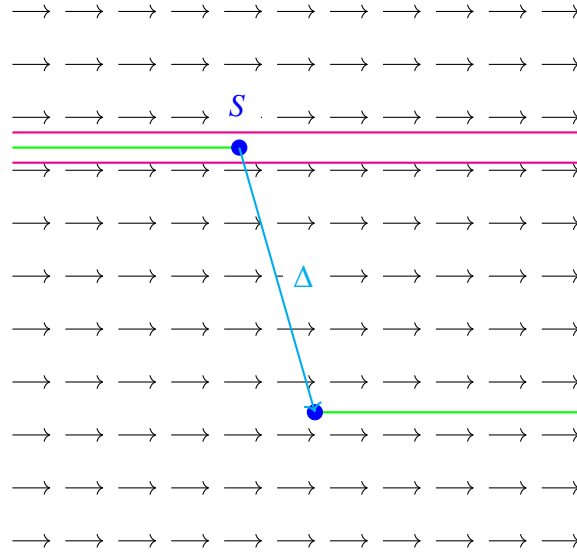
\begin{figure}
	\centering
	\begin{tikzpicture}
		\foreach \x in {0, 0.7,..., 7}{
			\foreach \y in {0, 0.7,..., 7}{
				\draw[->] (\x,\y) -- (\x+0.5,\y);
			}
		}
		\draw[thick, green] (0,5.2) -- (3,5.2);
		\draw[thick, green] (4,1.7) -- (7.5,1.7);
		\draw[thick, magenta] (0,5.4) -- (7.5,5.4);
		\draw[thick, magenta] (0,5) -- (7.5,5);
		\draw[blue, fill] (3,5.2) circle [radius=0.1];
		\draw[blue, fill] (4, 1.7) circle [radius=0.1];
		\draw[cyan, thick, ->] (3,5.2) -- (4, 1.7);
		\node[blue, fill=white] at (3,5.75) {$S$};
		\node[cyan, fill=white] at (3.85,3.5) {$\Delta$};
	\end{tikzpicture}
	\caption{When the $\mathrm{codim}(S)\geq 2$, resets become rare. In particular, generic trajectories will ``miss'' the guard. In this example, the guard is a single point in the plane and only a single arc resets - all others miss.}
	\label{fig:codim2}
\end{figure}
\begin{remark}
	A common definition of HDSs utilizes multiple state-spaces by incorporating a directed graph structure, cf., e.g., \cite{simic2001structural}. However, this version can still be encapsulated by Definition \ref{def:HDS} as $M$ is not assumed to be connected.
\end{remark}
The dynamics of $\mathcal{H}$ can be informally described as
\begin{equation}\label{eq:hybrid_dynamics}
	\begin{cases}
		\dot{x} = X(x), & x\not\in S, \\
		x^+ = \Delta(x^-), & x\in S,
	\end{cases}
\end{equation}
and its general solution will be described by its hybrid flow where completeness of its flow will be implicitly assumed. Throughout, it will be understood that a superscript of ``$+$'' indicated a limit from above and ``$-$'' as a limit from below.

\begin{definition}[Hybrid Flow]
	Let $\mathcal{H}=(M,S,X,\Delta)$ be an HDS. Let $\varphi:\mathcal{D}\to M$ be the flow for the continuous dynamics $\dot{x}=X(x)$ where $\mathcal{D}\subset \mathbb{R}\times M$ is the domain. The hybrid flow for the HDS \eqref{eq:hybrid_dynamics} will be denoted by $\varphi^\mathcal{H}:\mathcal{D}^\mathcal{H}\subset\mathcal{D}\to M$ and satisfies the following: 
	\begin{itemize}
		\item If for all $s\in[0,t]$, $\varphi(s,x)\not\in S$, then $\varphi^\mathcal{H}(t,x) = \varphi(t,x)$ ( outside of the guard, the hybrid flow obeys the continuous equations of motion).
		\item Suppose that there exists a time $t^*\in (0,t)$ such that $\varphi(t^*,x)\in S$. Then,
		\begin{equation*}
			\lim_{s\searrow t^*} \, \varphi^\mathcal{H}(s,x) = \Delta\left(
			\lim_{s\nearrow t^*} \, \varphi^\mathcal{H}(s,x)\right),
		\end{equation*}
		provided that $\Delta(\varphi(t^*,x))\not\in S$ (on the guard, the hybrid flow jumps according to the reset map $\Delta$).
	\end{itemize}
	Alternatively, the hybrid flow may be denoted by $\varphi_t^\mathcal{H}(x) = \varphi^\mathcal{H}(t,x)$ to emphasise the dependence on the initial condition.
\end{definition}
The above notion of a hybrid flow is limited to trajectories with a finite number of separated resets; for more on the solution concept of HDSs, see e.g. \cite{linan_controllability, teelSurvey, sergey2006impulsive}. The flow of a hybrid system is capable of pathological behavior which does not occur in continuous-time systems, namely, beating, blocking, and Zeno which will be discussed in Section \ref{sec:Zeno}. Trajectories that are free of these behaviors are called \textit{regular} (compare to regular HDSs).
\begin{definition}[Regular Hybrid Arc]
	A hybrid arc, $\gamma:[t_0,t_f]\to M$ where $\gamma(t) = \varphi^\mathcal{H}(t,x)$ for some $x\in M$, is called \textit{regular} if there exists a finite number of time instants where resets occur,
	\begin{equation*}
		\# \left\{ t\in [t_0,t_f] : \lim_{s\nearrow t} \varphi^\mathcal{H}(s,x)\in {S}\right\} < \infty,
	\end{equation*}
	and resets move the state away from the guard,
	\begin{equation*}
		\lim_{s\nearrow t} \gamma(s) \in {S} \implies \Delta\left( \lim_{s\nearrow t} \gamma(s) \right) \not\in {S}.
	\end{equation*}
\end{definition}
Even for the case of regular hybrid arcs, the trajectories will generally fail to be continuous at resets. This makes continuous (and smooth) dependence on initial conditions impossible. Fortunately, a weaker notion can be satisfied, the quasi-smooth dependence property depicted in Figure \ref{fig:QSDP}.
\begin{definition}[Quasi-Smooth Dependence Property]\label{def:qsdp}
	Consider an HDS, $\mathcal{H}$, with flow $\varphi^\mathcal{H}$. Then, $\mathcal{H}$ has the quasi-smooth dependence property if for every $x\in M\setminus {S}$ and $t\in \mathbb{R}$ such that $\varphi^\mathcal{H}(t,x)\not\in{S}$, there exists an open neighborhood $x\in U$ such that $U\cap {S}=\emptyset$ and the map $\varphi^\mathcal{H}(t,\cdot):U\to M$ is smooth. If $\mathcal{H}$ has the quasi-smooth dependence property, it will be referred to as smooth.
\end{definition}
\begin{figure}
	\centering
	\begin{tikzpicture}
		\draw[blue, thick] (0,0) to [out=65,in=-90] (1.5,3) to [out=90,in=-150] (3,5);
		\node[blue, above right] at (3,5) {${S}$};
		\draw[red, thick] (2.5,-0.5) to [out=100, in=-100] (4,2) to [out=80, in=180] (5.25,3.5);
		\node[red, right] at (5.25,3.5) {$\Delta({S})$};
		\draw[thick] (-1,3.25) to [out=10,in=200] (1.5,3);
		\draw[fill] (-1,3.25) circle [radius=0.04];
		\node[left] at (-1,3.25) {$x$};
		\draw[thick] (4,2) to [out=-30, in=110] (6,1);
		\draw[fill] (6,1) circle [radius=0.04];
		\node[below right] at (6,1) {$\varphi^\mathcal{H}(t,x)$};
		\draw[thick, dashed, ->] (1.5,3) -- (4,2);
		\node[above] at (2.75, 2.5) {$\Delta$};
		\draw[thick, dashed] (-1,3.25) circle [radius=0.75];
		\node at (-1.75, 4) {$U$};
		\path[draw, use Hobby shortcut, closed=true, dashed]
		(5.5,1) .. (6,1.7) .. (6.5,1.3) .. (7.75,1) .. (6,0.2);
		\node[above right] at (6.5,1.4) {$\varphi^\mathcal{H}(t, U)$};
	\end{tikzpicture}
	\caption{A hybrid system with the quasi-smooth dependence property. Although the flow fails to be smooth along the guard, the restriction to the open neighborhood $U$ is smooth when both $U$ and $\varphi^\mathcal{H}(t,U)$ are disjoint from $\mathcal{S}$.}
	\label{fig:QSDP}
\end{figure}
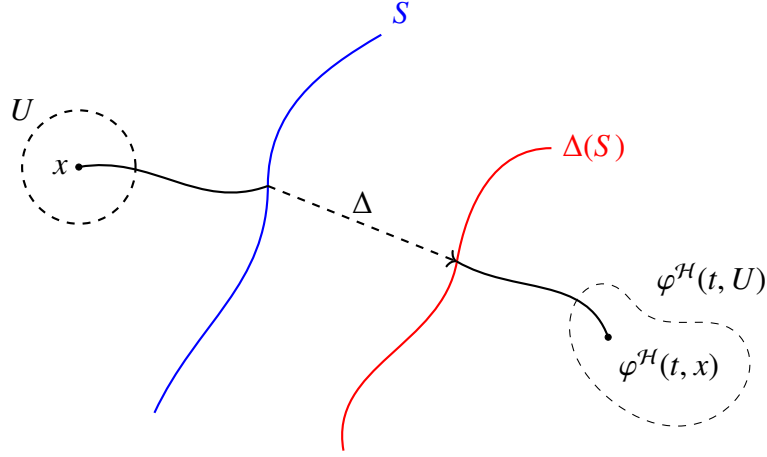
It turns out that many HDSs will have the quasi-smooth dependence property as it is related to transversality, cf. Theorem 1 in \cite{clark2023invariant}. We make the following assumption.
\begin{assume}\label{as:regular}
	Unless otherwise stated, all HDSs will have the quasi-smooth dependence property and all arcs will be regular.
\end{assume}

An analysis of the phenomena of beating and Zeno, where Assumption \ref{as:regular} is not satisfied, is reserved to Section \ref{sec:Zeno}.

A useful feature of the quasi-smooth dependence property is that it allows for differential forms to be preserved across resets in a meaningful way. For a continuous-time system, $\dot{x}=X(x)$, a differential form $\alpha$ is invariant if, and only if, its Lie derivative vanishes, $\mathcal{L}_X\alpha=0$. Likewise, for a discrete-time dynamical system, $x_{n+1}=f(x_n)$, a differential form is invariant if, and only if, $f^*\alpha=\alpha$. The conditions for a differential form to be invariant for a hybrid system are a mixture of the above conditions and are given below.
\begin{theorem}[\cite{clark2023invariant}]\label{thm:inv_form}
	A differential form $\alpha\in\Omega^*(M)$ is preserved along the hybrid flow if, and only if, $\mathcal{L}_X\alpha=0$ and
	\begin{equation}\label{eq:invariance}
		\begin{array}{cl}
			\Delta^*\alpha = \iota^*\alpha, & (\text{specular condition}) \\
			\Delta^*i_X\alpha = \iota^*i_X\alpha, & (\text{energy condition})
		\end{array}
	\end{equation}
	where $\iota:{S}\hookrightarrow M$ is the inclusion and $i_X\alpha = \alpha(X,\cdot)$ is interior multiplication. Let $\mathcal{A}_\mathcal{H}\subset\Omega^*(M)$ be the set of all invariant differential forms. Then $\mathcal{A}_\mathcal{H}$ forms a $\wedge$-subalgebra that is closed under $d$ and $i_X$.
\end{theorem}
\noindent The two conditions in \eqref{eq:invariance} are found via the \textit{augmented differential} which is the way to differentiate the reset map that is compatible with the vector field. See also the notion of a saltation matrix, e.g., \cite{saltation}.
\begin{definition}[Augmented Differential]\label{def:aug_dif}
	For a smooth map, $\Delta:{S}\to M$, and a vector field, $X\in\mathfrak{X}(M)$, the augmented differential of $\Delta$ is the linear map $\Delta_*^X:TM|_S\to TM$, defined via
	\begin{equation*}
		\begin{cases}
			\Delta_*^X\cdot u = \Delta_*\cdot u, & u\in T_x{S}\subset T_xM, \\
			\Delta_*^X\cdot X(x) = X\left(\Delta(x)\right).
		\end{cases}
	\end{equation*}
\end{definition}
\noindent The augmented differential is uniquely determined as long as $X(x)\not\in T_x{S}$ for all $x\in S$. Moreover, this map is invertible as long as $\Delta$ is an immersion and $X(y)\not\in T_y\Delta(S)$ for all $y\in\Delta(S)$.

An important consequence of Theorem \ref{thm:inv_form} is that Hamiltonian hybrid dynamical systems (those generated by the hybrid maximum principle below) preserve the symplectic form as well as the induced volume-form (see Section \ref{sec:impact_system}). Finding invariant volumes is important as it imposes strong restrictions on Zeno behavior, as will be seen in Section \ref{sec:Zeno}. 
\subsection{Impact systems}\label{sec:impact_system}
A distinguished subclass of HDSs are impact systems and serve as the Hamiltonian analog which will be central to the hybrid maximum principle. These are mechanical systems where physical impacts occur, e.g., a billiard ball ricocheting off a cushion.
Their definition is stated for mechanical Hamiltonian systems but extends naturally to Lagrangian systems.
\begin{definition}
	An impact system is a 3-tuple $\mathcal{I}=(Q,H,\Sigma)$ where
	\begin{enumerate}
		\item $Q$ is a smooth manifold,
		\item $H:T^*Q\to\mathbb{R}$ is a smooth Hamiltonian, and
		\item $\Sigma\subset Q$ is a smooth, embedded, co-dimension 1 submanifold.
	\end{enumerate}
\end{definition}
Generally speaking, the set $\Sigma$ cannot be described as a regular level-set of a function $h:Q\to\mathbb{R}$; $\Sigma$ need not even be orientable. However, this can always be done locally and assuming that $\Sigma = h^{-1}(0)$ will not cause problems. As such, an impact system may be referred to by $(Q,H,h)$, as is used in \cite{Ames06isthere}.

An impact system induces an HDS via the following procedure: $(Q,H,\Sigma)$ becomes $(M,{S},X,\Delta)$ where: $M = T^*Q$, $X = X_H$ is the Hamiltonian vector field given by $i_{X_H}\omega = dH$, ${S}$ is the set of ``outward pointing momenta,'' and
\begin{equation}\label{eq:outward_momenta}
	{S} = \left\{ (x,p) \in T^*Q|_\Sigma : \left(\pi_Q^*dh\right)(X_H) > 0\right\} \subset T^*Q|_\Sigma,
\end{equation}
where $\pi_Q:T^*Q\to Q$ is the canonical cotangent projection. The reset $\Delta$ will follow from variational principles discussed below. Notice that the guard, ${S}$, depends on the Hamiltonian and may be denoted by $S_H$ if the choice of Hamiltonian is unclear.  
\begin{remark}
	By defining the guard as outward pointing momenta in \eqref{eq:outward_momenta}, we are implicitly choosing an orientation for $\Sigma$. If $\Sigma$ is not orientable, $S$ can be constructed on the orientable double cover.
\end{remark}
\begin{example}\label{ex:resets}
	Let $Q = \{x^2+y^2\leq 1\}\subset \mathbb{R}^2$ be the unit disk and $\Sigma = \mathbb{S}^1$ the unit circle. Choose
	\begin{equation*}
		H_1 = p_yx - p_xy, \quad H_2 = p_x^2+p_y^2.
	\end{equation*}
	The induced guards are
	\begin{equation*}
		{S}_{H_1} = \emptyset, \quad {S}_{H_2} = \left\{ (x,y;p_x,p_y): x^2+y^2=1, \, xp_x+yp_y > 0\right\}.
	\end{equation*}
	See Figure \ref{fig:induced_guards} for an illustration for $S_{H_1}$ and $S_{H_2}$.
	\begin{figure}
		\centering
		\begin{subfigure}{0.45\textwidth}
			\begin{tikzpicture}
				\draw[->, thick] (-3.5,0) -- (3.5,0);
				\draw[->, thick] (0,-3) -- (0,3);
				\draw[thick, blue] (0,0) circle [radius=2.5];
				\foreach \theta in {0, 20,..., 360}{
					\foreach \r in {0.5, 1, 1.5, 2, 2.5, 3}{
						\pgfmathsetmacro{\x}{\r*cos(\theta)}
						\pgfmathsetmacro{\y}{\r*sin(\theta)}
						\pgfmathsetmacro{\vx}{-\y}
						\pgfmathsetmacro{\vy}{\x}
						\draw[->, red] (\x,\y) -- (\x+\vx/5, \y+\vy/5);
					}
				}
				\node[blue, fill=white] at (1.5,-1.5) {$\Sigma$};
			\end{tikzpicture}
		\end{subfigure}
		\begin{subfigure}{0.45\textwidth}
			\begin{tikzpicture}
				\draw[->, thick] (-3.5,0) -- (3.5,0);
				\draw[->, thick] (0,-3) -- (0,3);
				\draw[thick, blue] (0,0) circle [radius=2.5];
				\node[blue, fill=white] at (1.5,-1.5) {$\Sigma$};
				\draw[thick] (1.7678-1,1.7678+1) -- (1.7678+1,1.7678-1);
				\node[below left] at (1.7678,1.7678) {$(x,y)$};
				\draw[thick] (0.7678-1, 2.7678-1) -- (0.7678+1, 2.7678+1) -- (1.7678+2, 1.7678) -- (1.7678, 1.7678-2) -- (0.7678-1, 2.7678-1);
				\draw[pattern=horizontal lines, pattern color=red] (1.7678-1,1.7678+1) -- (0.7678+1, 2.7678+1) -- (1.7678+2, 1.7678) -- (1.7678+1,1.7678-1);
				\node[below right, red] at (1.7678+1.25,1.7678-0.75) {$S_{(x,y)}$};
				\draw[->, red, thick] (1.7678,1.7678) -- ((1.7678+0.25,1.7678+1);
				\draw[->, red, thick] (1.7678,1.7678) -- (1.7678+1.5,1.7678-0.25);
				\draw[fill=black] (1.7678,1.7678) circle [radius=0.1];
			\end{tikzpicture}
		\end{subfigure}
		\caption{Induced guards from Example \ref{ex:resets}. Left: A plot of the projected vector field $(\pi_Q)_*X_{H_1}$. As this vector field is nowhere transverse to the gaurd, the induced guard is empty, $S_{H_1}=\emptyset$. Right: A fiber of the induced guard, $S_{(x,y)}\subset T^*_{(x,y)}Q$, corresponding to $H_2$. The red hatched region corresponds to all outward pointing vectors rooted at the point $(x,y)\in\Sigma$.}
		\label{fig:induced_guards}
	\end{figure}
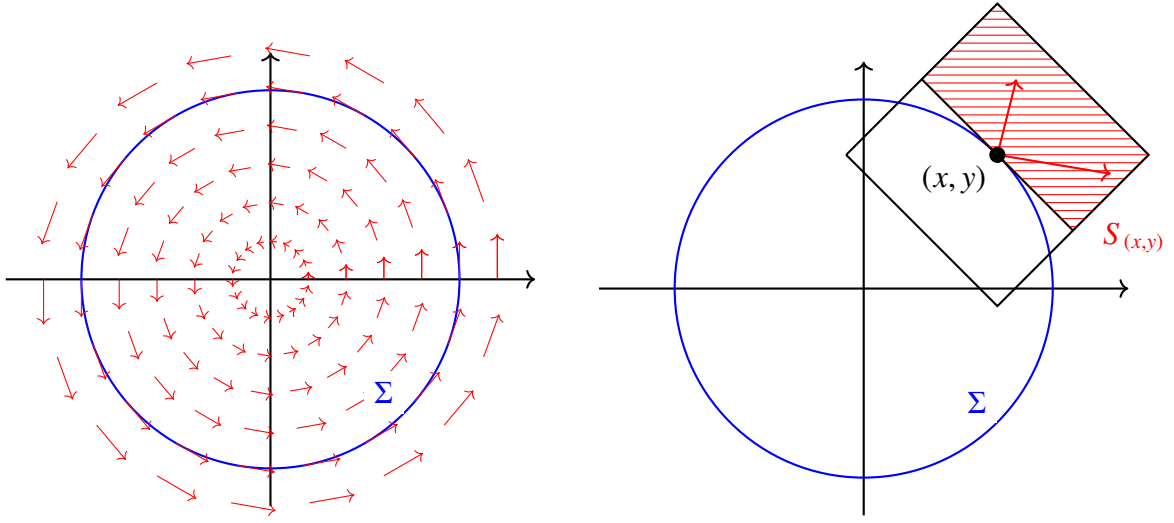
\end{example}

The reset map, $\Delta$, encodes the discontinuous transition at the point of impact. This is constructed via the following axiom; see Chapter 3 in \cite{brogliato}.
\begin{axiom}\label{ax:bounce}
	Mechanical impacts are the identity on the position variables, and the velocity/momentum variables change according to Hamilton's/Lagrange-d'Alembert's principle.
\end{axiom}
Applying this axiom, we end up with the usual Weierstrass-Erdmann corner conditions \cite{gelfand2012calculus}. In the Lagrangian formalism, they take the form
\begin{equation*}
	\begin{split}
		\left(\frac{\partial L}{\partial\dot{x}}^+ - \frac{\partial L}{\partial \dot{x}}^-\right)\delta x &= 0, \\
		-\left( E_L^+ - E_L^-\right)\delta t &= 0,
	\end{split}
\end{equation*}
where $E_L$ is the energy associated to the Lagrangian and the superscripts denote pre- and post-reset states. Likewise, in the Hamiltonian formalism, the corner conditions are
\begin{equation}\label{eq:H_corner_conditions}
	\begin{split}
		\left( p^+ - p^-\right)\delta x &= 0, \\
		-\left( H^+ - H^-\right)\delta t &= 0,
	\end{split}
\end{equation}
where $p$ is the momentum. At the time of impact, $\delta t$ is arbitrary while $\delta x\in TS$ is constrained. Therefore, the corner conditions specialize to
\begin{equation}\label{eq:corner_conditions}
	\begin{split}
		p^+ &= p^- + \varepsilon\cdot dh, \\
		H^+ &= H^-,
	\end{split}
\end{equation}
where $\varepsilon$ is an unknown multiplier to enforce conservation of energy.
These conditions naturally explain elastic impacts as energy is conserved and the change in momentum is proportional to the normal of the impact surface. An intrinsic formulation is presented below.
\begin{theorem}\label{thm:corner}
	The corner conditions \eqref{eq:corner_conditions} are equivalent to
	\begin{equation}\label{eq:impact_conditions}
		\overline{\Delta}^*\vartheta_H = \iota^*\vartheta_H,
	\end{equation}
	where $\iota:\mathbb{R}\times {S}\hookrightarrow \mathbb{R}\times T^*Q$ is the inclusion, $\overline{\Delta}(t,x) = (t,\Delta(x))$, $\vartheta_H$ is the action form, and
	\begin{equation*}
		\vartheta_H := p_i\cdot dx^i - H\cdot dt\in\Omega^1(\mathbb{R}\times T^*Q).
	\end{equation*}
\end{theorem}
\begin{proof}
	Choose local coordinates such that the impact occurs when the last coordinate vanishes, $\Sigma = \{x^n=0\}$. Then, \eqref{eq:impact_conditions} is
	\begin{equation*}
		\left(p_i\circ\Delta\right) dx^i - \left(H\circ\Delta\right) dt = p_idx^i - Hdt, \quad i=1,\ldots,n-1.
	\end{equation*}
	Equating terms, we see that all $p_i$ must remain fixed, with the exception of $p_n$, and $H$ must also remain fixed. This is precisely \eqref{eq:corner_conditions}.
\end{proof}
\begin{remark}
	The impact condition \eqref{eq:impact_conditions} also describes the impacts even when the surface is moving, e.g., a tennis racket striking the tennis ball. Let $\Sigma_t\subset Q$ be the time-dependent surface and suppose that it is described by the level-set of a function $h_t(x) = h(t,x)=0$. Then, \eqref{eq:impact_conditions} dictates that the impact equations are given by
	\begin{equation}\label{eq:time_dependent_corner}
		\begin{split}
			\left( p_i^+ - p_i^-\right) &= \varepsilon\cdot \frac{\partial h}{\partial x^i},\\
			-\left( H^+ - H^-\right) &= \varepsilon \cdot \frac{\partial h}{\partial t}.
		\end{split}
	\end{equation}
\end{remark}

\begin{example}[Bouncing Ball]\label{ex:bball}
	Consider the bouncing ball on an oscillating table, which will serve as a running example throughout. The continuous-time dynamics are given by
	\begin{equation}\label{eq:bball_continuous}
		\dot{x} = \frac{1}{m}y, \quad \dot{y} = -mg,
	\end{equation}
	where $m$ is the mass of the ball, $g$ is the acceleration due to gravity, $x$ is the vertical height of the ball, and $y$ is the ball's momentum. 
	In the simple case where the table is stationary and flat, the impact occurs when $x=0$ and $y\mapsto -y$ is the reset. However, in the case where the table is vertically oscillating, the impact occurs when $x= A\sin\omega t$. From \eqref{eq:impact_conditions}, the reset equations are
	\begin{equation}\label{eq:bball_impact_oscillating}
		\begin{split}
			x & \mapsto x \\
			y & \mapsto -y + 2mA\omega\cos(\omega t),
		\end{split}
	\end{equation}
	which agrees with the usual impact relationship, cf. \S 2.4 in \cite{guckenheimer1983}.
\end{example}
We conclude this section with the important observation that hybrid systems generated by impact systems are symplectic; for a discussion on the symplecticity of impact systems, cf. \cite{fetecau2003}.
\begin{theorem}\label{thm:sympectic_invariant}
	Time-independent impact systems preserve the symplectic form as well as the volume form $\omega^n$. Time-dependent impact systems preserve the form $\omega_H :=-d\vartheta_H$ as well as the volume form $dt\wedge \omega_H^n$.
\end{theorem}
\begin{proof}
	We start with the time-independent case. The energy condition is satisfied via conservation of energy:
	\begin{equation*}
		\Delta^*i_{X_H}\omega = \Delta^*dH = \iota^*dH = \iota^*i_{X_H}\omega.
	\end{equation*}
	The specular condition follows from
	\begin{equation*}
		\Delta^*\omega = -\Delta^*d\vartheta = -d\Delta^*\vartheta = -d\iota^*\vartheta = \iota^*\omega.
	\end{equation*}
	We next consider the time-dependent case. The energy condition is trivially satisfied as $i_{X_H}\omega_H=0$. The specular condition follows from
	\begin{equation*}
		\overline{\Delta}^*\omega_H = -d\overline{\Delta}^*\vartheta_H = -d\iota^*\vartheta_H = \iota^*\omega_H.
	\end{equation*}
	The invariant volumes follow as the wedge product of invariant forms remains invariant.
\end{proof}

\section{Hybrid control systems}\label{sec:problem}
A hybrid control system is an extension of HDSs in Definition \ref{def:HDS}. A similar object, \textit{geometric hybrid systems}, along with a study of their solutions and controllability, can be found in \cite{linan_controllability}. In our definition of a hybrid control system, it is assumed that the controls only influence the continuous dynamics and neither the guard nor the reset.
\begin{definition}
	A hybrid control system is a 5-tuple $\mathcal{HC}=(M,\mathcal{U},\Sigma,f,\Delta)$ such that
	\begin{enumerate}
		\item $M$ is a smooth (finite-dimensional, $\dim(M)=n$) manifold,
		\item $\Sigma\subset M$ is a smooth embedded submanifold with co-dimension 1,
		\item $\mathcal{U}\subset \mathbb{R}^m$ is a closed subset consisting of admissible controls,
		\item $f:M\times V\to TM$ is smooth where $\mathcal{U}\subset V$ is an open neighborhood, and
		\item $\Delta:\Sigma\to M$ is a smooth map such that $\Delta(\Sigma)$ is also a smooth embedded submanifold with a continuous extension $\overline{\Delta}:\overline{\Sigma}\to M$.
	\end{enumerate}
	Moreover, a regular hybrid control system additionally satisfies
	\begin{enumerate}
		\item[6.] $\Sigma\cap\Delta(\Sigma)=\emptyset$ and $\overline{\Sigma}\cap\overline{\Delta(\Sigma)}\subset M$ has co-dimension at least 2.
	\end{enumerate}
\end{definition}
The set $\mathcal{U}$ contains all admissible control values while control signals lie within the set $\mathcal{U}^{[t_0,t_f]}$ of (measurable) functions $[t_0,t_f]\to\mathcal{U}$. Using the notation of $\Sigma$ rather than $S$ will be apparent in Definition \ref{def:extended_rm} below.

For a given hybrid control system, an optimal control problem can be imposed by introducing the \textit{cost functional} subject to a running and terminal cost
\begin{gather*}
	J: M\times \mathcal{U}^{[t_0,t_f]}\times [t_0,t_f] \to \mathbb{R} \\
	J(x_0, u(\cdot), s) = \int_s^{t_f} \, \ell(x(t),u(t))\, dt + g\left( x(t_f)\right),
\end{gather*}
subject to the controlled dynamics
\begin{equation*}
	\begin{cases}
		\dot{x}(t) = f(x(t),u(t)), & x(t)\not\in \Sigma, \\
		x(t)^+ = \Delta(x(t)^-), & x(t)\in \Sigma,
	\end{cases}
\end{equation*}
and the initial condition $x(t_0) = x_0$. Solutions to the optimal control problem are given by
\begin{equation*}
	u^*_{x_0}(\cdot) = \arg \min_{u\in\mathcal{U}^{[t_0,t_f]}} \, J(x_0, u(\cdot), 0),
\end{equation*}
assuming they exist. The cost associated to the optimal control is called the \textit{value function} and is 
\begin{gather*}
	J^*:M\times [t_0,t_f] \to \mathbb{R} \\
	J^*(x_0,s) := J(x_0, u_{x_0}^*(\cdot), s).
\end{gather*}

An alternative optimal control problem to the terminal cost above is the fixed end point problem. Both of these problems can be addressed by the hybrid maximum principle and are explicitly stated below.
\begin{problem}[Terminal cost]\label{prob:TI_TC}
	The cost functional to be minimized for systems subject to a terminal cost, $g:M\to\mathbb{R}$, is given by,
	\begin{gather*}
		J:M\times \mathcal{U}^{[t_0,t_f]}\times [t_0,t_f]\to\mathbb{R}, \\
		J(x_0,u(\cdot),s) = \int_s^{t_f} \, \ell(x(t),u(t))\, dt + g(x(t_f)),
	\end{gather*}
	subject to $x(t_0)=x_0$ and the dynamics
	\begin{equation*}
		\begin{cases}
			\dot{x}(t) = f(x(t),u(t)), & x(t)\not\in \Sigma, \\
			x(t)^+ = \Delta(x(t)^-), & x(t)\in \Sigma.
		\end{cases}
	\end{equation*}
\end{problem}
\begin{problem}[Fixed end points]\label{prob:TI_FE}
	The cost functional to be minimized for systems subject to fixed end points is given by,
	\begin{gather*}
		J:M\times \mathcal{U}^{[t_0,t_f]}\times [t_0,t_f]\to\mathbb{R}, \\
		J(x_0,u(\cdot),s) = \int_s^{t_f} \, \ell(x(t),u(t))\, dt,
	\end{gather*}
	subject to $x(t_0)=x_0$, $x(t_f)=x_f$, and the dynamics
	\begin{equation*}
		\begin{cases}
			\dot{x}(t) = f(x(t),u(t)), & x(t)\not\in \Sigma, \\
			x^+ = \Delta(x^-), & x(t)\in \Sigma.
		\end{cases}
	\end{equation*}
\end{problem}

\section{Necessary conditions for optimality}\label{sec:necessary}
For continuous-time systems, the value function satisfies the Hamilton-Jacobi-Bellman (HJB) partial differential equation (PDE). Solving this PDE via the method of characteristics results in the PMP. In the hybrid setting, the value function will satisfy the \textit{hybrid} HJP PDE which will result in the HPMP.

Away from resets, the optimality conditions will be identical to the usual continuous ones. At a point of reset, the value function undergoes a discontinuity and can be described by an \textit{extended reset map}.
Let $\mathcal{HC} = (M,\mathcal{U},\Sigma,f,\Delta)$ be a hybrid control system with cost functional $J$. Define the Hamiltonian $\hat{H}:T^*M\times\mathcal{U}\times \mathbb{R}\to\mathbb{R}$ where
\begin{equation}\label{eq:formH}
	\hat{H}(x,p,u,p_0) = p_0\ell(x,u) + \langle p,f(x,u)\rangle,
\end{equation}
and the optimal Hamiltonian $H:T^*M\times \mathbb{R}\to\mathbb{R}$, 
\begin{equation}\label{eq:minH}
	H(x,p,p_0) = \min_u \, \hat{H}(x,p,u,p_0).
\end{equation}

We will turn the hybrid control system on $M$ into a Hamiltonian hybrid system on $T^*M$ . In order to do so, we need to specify the vector field, the guard and the reset map of the new system. The vector field will be given by the Hamiltonian vector field arising from the optimal Hamiltonian \eqref{eq:minH}. The \textit{extended guard} contains the set of ``outward pointing momenta" as in Section \ref{sec:impact_system}. The  ``Hamiltonian jump condition'' \eqref{eq:impact_conditions}  will produce the \textit{extended reset map}.
\begin{definition}[Extended Guard and Reset Map]\label{def:extended_rm}
	Let $\mathcal{HC}=(M,\mathcal{U},\Sigma,f,\Delta)$ be a hybrid control system and $H:T^*M\to\mathbb{R}$ an optimal Hamiltonian. Then, the set 
	\begin{equation}\label{eq:extended_guard}
		S = \left\{ (x,p)\in T^*M|_\Sigma : \pi_M^*dh(X_H) > 0\right\} \subset T^*M|_\Sigma,
	\end{equation}
	is called the \textit{extended guard} where $\Sigma = h^{-1}(0)$ and $\pi_M:T^*M\to M$ is the cotangent bundle projection. A smooth map $\tilde{\Delta}:S\to T^*M$ is called the \textit{extended reset} if it satisfies
	\begin{equation}\label{eq:controlled_impact}
		\left( \mathrm{Id}\times\tilde{\Delta}\right)^*\vartheta_H = \iota^*\vartheta_H, \quad
		\begin{tikzcd}
			\mathbb{R}\times S \arrow[r, "\mathrm{Id}\times\tilde{\Delta}"] \arrow[d, "\mathrm{Id}\times\pi_M"] & \mathbb{R}\times T^*M \arrow[d, "\mathrm{Id}\times\pi_M"] \\
			\mathbb{R}\times \Sigma \arrow[r,"\mathrm{Id}\times\Delta"] & \mathbb{R}\times M
		\end{tikzcd}
	\end{equation}
	such that the diagram is commutative, and $\vartheta_H$ is the action form
	\begin{equation*}
		\vartheta_H = p_i\cdot dx^i - H\cdot dt \in \Omega^1(\mathbb{R}\times T^*M).
	\end{equation*}
\end{definition}
The extended reset \eqref{eq:controlled_impact} is a natural extension of the corner conditions \eqref{eq:impact_conditions} and can be expressed in coordinates by
\begin{equation*}
	\begin{split}
		p^+\circ\Delta_* &= p^- + \varepsilon\cdot dh, \\
		H\circ\tilde{\Delta} &= H.
	\end{split}
\end{equation*}

Existence/uniqueness of solutions to \eqref{eq:controlled_impact} is not guaranteed. In particular, if $\Delta$ fails to be immersive, the corner conditions become over-determined. Fortunately, consistency conditions can be found to make \eqref{eq:controlled_impact} neither over- nor under-determined \cite{cos_subm}. The type of HDS described above is called a Hamiltonian HDS. This notion is stricter than is often studied, e.g., \cite{leo_integrable}.

\begin{definition}[Hamiltonian HDS]
	An HDS, $\mathcal{H}=(N, S, X, \tilde{\Delta})$, is Hamiltonian if 
	\begin{enumerate}
		\item $N = T^*M$ is a cotangent bundle,
		\item $X = X_H$ is a Hamiltonian vector field for a Hamiltonian $H:T^*M\to\mathbb{R}$,
		\item There exists an embedded co-dimension 1 submanifold $\Sigma\subset M$ such that $S$ is given by \eqref{eq:extended_guard},
		\item There exists a smooth map $\Delta:S\to M$ such that $\tilde{\Delta}$ satisfies \eqref{eq:controlled_impact}.
	\end{enumerate}
\end{definition}
An impact system is a special type of Hamiltonian HDS where $\Delta=\mathrm{Id}$.

The necessary conditions for optimality are summarized in the theorem below. 

\begin{theorem}[Hybrid Maximum Principle, \cite{liberzon_oc}]\label{thm:HPMP}
	Let $\mathcal{HC} = (M, \mathcal{U}, \Sigma, f, \Delta)$ be a hybrid control system and let $u^*:[t_0,t_f]\to\mathcal{U}$ be an optimal control minimizing either Problem \ref{prob:TI_TC} or Problem \ref{prob:TI_FE} with corresponding state trajectory $x^*:[t_0,t_f]\to M$. If $x^*$ is a regular hybrid arc, then there exists a piecewise smooth curve $p^*:[t_0,t_f]\to T^*M$ and a constant $p_0\geq 0$ such that $(p_0,p^*(t)) \ne (0,0)$ for all $t\in [t_0,t_f]$ along with
	\begin{enumerate}
		\item Continuous flow: $x^*$ and $p^*$ satisfy Hamilton's equations when $(x^*,p^*)\not\in S$:
		\begin{equation*}
			\dot{x}^* = \frac{\partial H}{\partial p}, \quad \dot{p}^* = -\frac{\partial H}{\partial x},
		\end{equation*}
		where the Hamiltonian has the form
		\begin{gather*}
			H:T^*M\times \mathcal{U}\times \mathbb{R}^+ \to \mathbb{R} \\
			H(x,p,u,p_0) = p_0\ell(x,u) + \langle p, f(x,u)\rangle.
		\end{gather*}
		\item Jump condition: Let $\tau \in [t_0,t_f]$ where $(x^*(\tau^-),p^*(\tau^-))\in S$ is a reset time. Then
		\begin{equation*}
			(x^*(\tau^+),p^*(\tau^+)) = \tilde{\Delta}\left( (x^*(\tau^-),p^*(\tau^-))\right).
		\end{equation*}
		\item Minimization condition: For all $t\in [t_0,t_f]$ and $w\in\mathcal{U}$,
		\begin{equation*}
			H(x^*(t),p^*(t), u^*(t), p_0) \leq H(x^*(t),p^*(t), w, p_0).
		\end{equation*}
		\item Boundary conditions: For Problem \ref{prob:TI_TC}, the curves satisfy the boundary conditions
		\begin{equation*}
			x^*(t_0) = x_0, \quad p^*(t_f) = dg_{x^*(t_f)}.
		\end{equation*}
		For Problem \ref{prob:TI_FE}, the curves satisfy the boundary conditions
		\begin{equation*}
			x^*(t_0) = x_0, \quad x^*(t_f) = x_f.
		\end{equation*}
	\end{enumerate}
\end{theorem}

\begin{proof}
	The proof for quasi-smooth regular arcs closely follow the arguments presented in the book \cite{liberzon_oc} and can also be found in the thesis \cite{pakniyatthesis}. The principal point of departure between the classical maximum principle and its hybrid version is the jump map \eqref{eq:controlled_impact}. For exposition, we will show that \eqref{eq:controlled_impact} follows from the adjoint of the hybrid variational equation. Consider the augmented control system
	\begin{gather*}
		(x^0, x) \in \mathbb{R}\times M \\
		\dot{x}^0 = \ell(x,u), \quad \dot{x} = f(x,u)
	\end{gather*}
	subject to the initial conditions $x^0(t_0) = 0$ and $x(t_0) = x_0$. The augmented dynamics will be denoted by
	\begin{equation*}
		\dot{y} = g(y,u) = \begin{bmatrix}
			\ell(x,u) \\ f(x,u)
		\end{bmatrix}, \quad y = (x^0,x).
	\end{equation*}
	Variations, $(\delta x^0, \delta x)$,  propagate under the hybrid variational equation \cite{footslip}:
	\begin{equation}\label{eq:hybridVE}
		\begin{cases}
			\delta \dot{y} = \dfrac{\partial g}{\partial y}\cdot \delta y , & t\in\mathcal{T}, \\
			\delta y^+ = \Delta_*^g\cdot\delta y^-, & t\in\mathcal{T},
		\end{cases}
	\end{equation}
	where
	\begin{equation*}
		\mathcal{T} := \left\{ t\in [t_0,t_f] : \lim_{s\nearrow t} p^*(s)\in {S}\right\}
	\end{equation*}
	is the (finite) set of reset times.
	The augmented differential, $\Delta_*^g$, is given by
	\begin{equation}\label{eq:oc_variational}
		\begin{cases}
			\Delta_*^g\cdot \iota_*v = \iota_*\Delta_*v, & v\in T\mathcal{S}, \\
			\Delta_*^g\cdot g\left( y^-, u^-\right) = g\left(y^+, u^+\right),
		\end{cases}
	\end{equation}
	where $\iota:\Sigma\hookrightarrow \mathbb{R}\times M$ is the embedding. Choosing local coordinates such that $\Sigma = \{x^n=0\}$, the differential has the form
	\begin{equation*}
		\Delta_*^g = \begin{bmatrix}
			1 & 0 & A \\ 0 & \Delta_* & B \\ 0 & 0 & C
		\end{bmatrix},
	\end{equation*}
	where $A$, $B$, and $C$ are found such that the second condition in \eqref{eq:oc_variational} is satisfied.
	The co-states, $\bar{p} = (p_0,p)\in T^*\left(\mathbb{R}\times M\right)$, are the dual variables to the variations and evolve according to the adjoint equation:
	\begin{equation}\label{eq:hybridAE}
		\begin{cases}
			\dot{\bar{p}} = -\bar{p} \cdot \dfrac{\partial g}{\partial y}, & t\not\in\mathcal{T}, \\
			\bar{p}^+ = \bar{p}^-\cdot \left(\Delta_*^g\right)^{-1}, & t\in\mathcal{T}.
		\end{cases}
	\end{equation}
	The continuous component yields the usual adjoint process:
	\begin{equation*}
		\dot{p}_0 = 0, \quad \dot{p} = -p_0\frac{\partial \ell}{\partial x} - p\cdot \frac{\partial f}{\partial x},
	\end{equation*}
	while the jump map has $\bar{p}^+ \cdot \Delta_*^g = \bar{p}^-$, which yields
	\begin{equation*}
		p_0^+ = p_0^-, \quad p^+\circ\Delta_* = p^- + \varepsilon\cdot dh.
	\end{equation*}
	Energy conservation comes from the second condition in \eqref{eq:oc_variational} as
	\begin{equation*}
		\begin{split}
			H^- = \langle \bar{p}^-, g^-\rangle &= \langle \bar{p}^+\cdot\Delta_*^g, g^- \rangle \\
			&= \langle \bar{p}^+, \Delta_*^g\cdot g^- \rangle = \langle \bar{p}^+, g^+\rangle = H^+.
		\end{split}
	\end{equation*}
	Therefore, the map $p^-\mapsto p^+$ satisfies \eqref{eq:controlled_impact}.
\end{proof}

A nice corollary of the above construction of the co-state jump is its existence/uniqueness \textit{backward}. Unfortunately, the forward map may fail to be well-defined. This will be partially analyzed in Section \ref{sec:Zeno}.
\begin{corollary}
	Suppose that the flow intersects the guard transversely at a point of reset, i.e., for $x(t)\in \Sigma$, $f(x(t),u(t))\not\in T\Sigma$. Then, the map $p^+\mapsto p^-$ exists and is unique.
\end{corollary}
\begin{proof}
	The map $p^+\mapsto p^-$ is given by $\bar{p}^- = \bar{p}^+\cdot\Delta_*^g$,
	where $\bar{p} = (p_0,p)$. The augmented differential is well-defined as long as the flow is transverse to the guard.
\end{proof}

Although all of the data was assumed to be time-invariant, the above procedure remains valid for time-dependent systems as the (time-dependent) jumps remain subject to \eqref{eq:controlled_impact}.

For both time-independent and time-dependent problems, Theorem \ref{thm:sympectic_invariant} still holds and the proof is identical. For time-independent systems, energy conservation is guaranteed through the corner conditions. The following theorem deals with volume conservation. 

\begin{theorem}\label{thm:controlled_symplectic_invariant}
	Time-independent Hamiltonian HDSs preserve the symplectic form, and time-dependent Hamiltonian HDSs preserve the form $\omega_H$. Consequently, both systems are volume-preserving.
\end{theorem}


\section{Hybrid Lagrangian submanifolds}\label{sec:hybrid_Lagrange}
In this section, it will be assumed that the map $\tilde{\Delta}:S\to T^*M$ is well-defined and injective (so that it is invertible on its image). This assumption will be dropped in Section \ref{sec:Zeno}.

The necessary conditions in the hybrid maximum principle, Theorem \ref{thm:HPMP}, specify terminal conditions on the state and co-state; for problems with free end points but with terminal cost, this manifests as $p(t_f) = dg_{x_f}$, while for fixed end points, the terminal condition is $x(t_f) = x_f$.
These end conditions state that the terminal conditions must lie within a specified submanifold of $T^*M$
\begin{equation*}
	\begin{array}{lr}
		p(t_f) \in \Gamma_{dg} = \left\{ (x,dg_x): x\in M\right\}, & \text{(Problem \ref{prob:TI_TC})} \\
		p(t_f) \in T_{x_f}^*M, & \text{(Problem \ref{prob:TI_FE})}
	\end{array}
\end{equation*}
Both of these submanifolds are examples of \textit{Lagrangian submanifolds} in the symplectic manifold $T^*M$. In this study of Lagrangian submanifolds, we will always assume that the ambient symplectic manifold is a cotangent bundle with the canonical symplectic form $\omega$.

\begin{definition}[Lagrangian Submanifolds]
	A submanifold $\mathcal{L}\subset T^*M$ is a Lagrangian submanifold if $\omega|_\mathcal{L}=0$ and $\dim\mathcal{L} = \dim M$.
\end{definition}
Both of the terminal manifolds for the optimal control are Lagrangian via the following well-known proposition.
\begin{proposition}[5.3.15 in \cite{abrahammarsden}]
	For a cotangent bundle $T^*M$ with the canonical symplectic structure, the following submanifolds are Lagrangian:
	\begin{enumerate}
		\item $T_x^*M$ for any $x\in M$,
		\item $M$ viewed as the zero-section, and
		\item $\Gamma_\alpha = \left\{ (x,\alpha_x): x\in M \right\}$ for a closed 1-form $\alpha$.
	\end{enumerate}
\end{proposition}
Lagrangian submanifolds propagate well under the flow of (non-hybrid) Hamiltonian systems.
\begin{proposition}[5.3.32 in \cite{abrahammarsden}]\label{prop:classical_lagrangian}
	Let $\mathcal{L}\subset T^*M$ be a Lagrangian submanifold and let $H:T^*M\to\mathbb{R}$ a Hamiltonian with flow $\varphi_t$. Then,
	\begin{enumerate}
		\item $\varphi_t(\mathcal{L})$ remains a Lagrangian submanifold for all $t$, and
		\item if the submanifold has constant energy, i.e., $H|_\mathcal{L}$ is constant, then it is invariant under the flow, $\varphi_t(\mathcal{L}) = \mathcal{L}$.
	\end{enumerate}
\end{proposition}
Extending this idea to hybrid systems results in ``hybrid Lagrangian submanifolds,'' which was introduced in \cite{clarkoprea} and is related to integrability of impact systems \cite{leo_integrable}.
\begin{definition}[Hybrid Lagrangian Submanifolds]
	Let $\mathcal{H} = (T^*M,S,X_H,\tilde{\Delta})$ be a hybrid Hamiltonian system where $\tilde{\Delta}$ satisfies \eqref{eq:controlled_impact}. A submanifold (with boundary) $\mathcal{L}\subset T^*M$ is a hybrid Lagrangian submanifold if
	\begin{enumerate}
		\item $\mathcal{L}\setminus\partial\mathcal{L}\subset T^*M$ is a Lagrangian submanifold, and
		\item $\partial\mathcal{L}\subset S\cup\tilde{\Delta}(S)$ such that $\tilde{\Delta}(\partial\mathcal{L}\cap S) = \partial\mathcal{L}\cap\tilde{\Delta}(S)$.
	\end{enumerate}
\end{definition}
\begin{figure}
	\centering
	\begin{tikzpicture}
		\draw[<->,black] (-1,0) -- (5,0);
		\draw[<->,black] (0,-1) -- (0,4);
		\node[below] at (5,0) {$M$};
		\node[left] at (0,4) {$T^*M$};
		\draw[blue] (1,-1) -- (1,4);
		\draw [fill] (1,0) circle [radius=0.07];
		\node[below right] at (1,0) {$\Sigma$};
		\node[right, blue] at (1,4.2) {$S\subset \pi_M^{-1}(\Sigma)$};
		\draw[red] (4,-1) -- (4,4);
		\draw [fill] (4,0) circle [radius=0.07];
		\node[below left] at (4,0) {$\Delta(\Sigma)$};
		\node[right, red] at (4,4.2) {$\tilde{\Delta}(S)\subset \pi_M^{-1}(\Delta(\Sigma))$};
		\draw[dotted,thick] (3,4) to [out=230,in=30] (2.75,3.7);
		\draw[thick] (2.75,3.7) to [out=220,in=40] (1,3);
		\draw[thick] (4,3) to [out=250, in=120] (1.75,2.15) to [out=300, in=90] (3.3,1.75) to [out=270, in=45] (1,1.25);
		\draw[thick] (4,1) to [out=230, in=10] (2.5,0.5);
		\draw[thick,dotted] (2.5,0.5) to [out=10,in=45] (2,0.25);
		\node[above] at (2.5,2) {$\mathcal{L}$};
		\draw[->,thick,dashed] (1,3) to [out=-130,in=45] (4,3);
		\draw[->,thick,dashed] (1,1.25) to [out=-130,in=45] (4,1);
	\end{tikzpicture}
	\caption{An example of hybrid Lagrangian $\mathcal{L}\subset T^*M$ \cite{clarkoprea}.}
	\label{fig:hybrid_Lagrangian}
\end{figure}
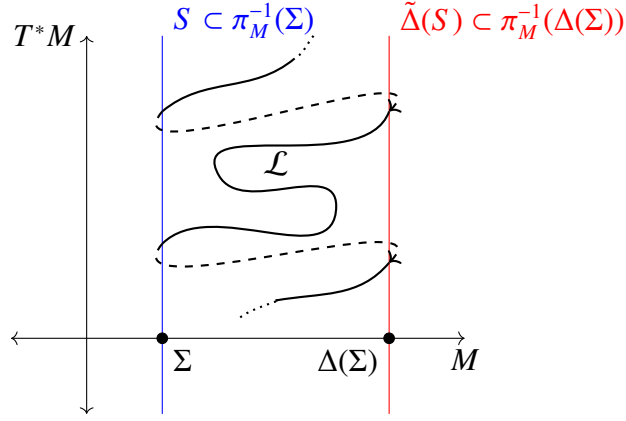
Proposition \ref{prop:classical_lagrangian} follows from the fact that the flow of a Hamiltonian system preserves the symplectic form. As this remains the case for impact systems via Theorem \ref{thm:controlled_symplectic_invariant}, an analogous result holds for impact systems.
\begin{proposition}
	Let $\mathcal{L}\subset T^*M$ be a hybrid Lagrangian submanifold and let $\varphi_t^\mathcal{H}$ be the flow of the Hamiltonian hybrid system. Then,
	\begin{enumerate}
		\item for all $t$, $\varphi_t^\mathcal{H}(\mathcal{L})$ is a hybrid Lagrangian submanifold, and
		\item if $\mathcal{L}$ has constant energy, then $\varphi_t^\mathcal{H}(\mathcal{L}) = \mathcal{L}$.
	\end{enumerate}
\end{proposition}
\begin{proof}
	This follows directly from the fact that the hybrid flow preserves the symplectic form.
\end{proof}
\subsection{Lagrangian intersections}\label{subsec:intersect}
The hybrid maximum principle states that necessary conditions for an optimal solution must satisfy the boundary conditions
\begin{equation*}
	p(t_0) \in T_{x_0}^*M, \quad p(t_f)\in\mathcal{L},
\end{equation*}
where $\mathcal{L} = \Gamma_{dg}$ or $T_{x_f}^*M$, depending on the problem. All of these sets are hybrid Lagrangian submanifolds, and a solution to $p(t_0)$ can be found by considering the intersection
\begin{equation}\label{eq:Lagrangian_intersection}
	p(t_0) \in T_{x_0}^*M \cap \varphi_{-t_f}^\mathcal{H}(\mathcal{L}),
\end{equation}
which is an intersection of hybrid Lagrangian submanifolds. A reason why the hybrid maximum principle is only necessary rather than sufficient is that this intersection may contain multiple solutions (which correspond to critical points of the optimal control problem). A geometric study of this intersection in the continuous-time domain is performed in \cite{lagrangian_manifolds}.
An analogous study for hybrid systems is a subject of future work.
\subsection{Caustics}\label{subsec:caustics}
Given a function $g:M\to\mathbb{R}$, it follows that $\varphi_t^\mathcal{H}(\Gamma_{dg})$ is a (hybrid) Lagrangian submanifold. A natural question to ask is whether or not this is also a graph of an exact (or closed) 1-form. This fails at caustics.
\begin{definition}[Hybrid caustics]
	The hybrid caustic of $\mathcal{L}\subset T^*M$ is the critical value set of $\pi_M:\mathcal{L}\to M$. This set is denoted by $\mathcal{C}(\mathcal{L})\subset M$. For a hybrid Hamiltonian flow $\varphi_t^\mathcal{H}$, the caustic trajectory is given by
	\begin{equation*}
		\mathcal{CT}(\mathcal{L}) := \left\{(t,x): x\in \mathcal{C}\left( \varphi_t^\mathcal{H}(\mathcal{L})\right) \right\} = \bigcup_t \, \mathcal{C}\left( \varphi_t^\mathcal{H}(\mathcal{L})\right).
	\end{equation*}
\end{definition}
When the original Lagrangian submanifold is a cotangent fiber, $\mathcal{L} = T_{x}^*M$, the notion of a caustic is closely related to conjugate points, cf. \cite{ascontrol}. Intuitively, a caustic (at a time $t$) is a point  where the Lagrangian submanifold $\mathcal{L}_t:= \varphi_t^\mathcal{H}(\mathcal{L})$ folds on itself, and, thus, the Hybrid maximum principle becomes a necessary, but not sufficient condition for optimality. The caustic trajectory is the union of all caustic points over all times.

\begin{definition}[Hybrid conjugate point]
	The point $(t_1, x_1)\in \mathbb{R}\times M$ is hybrid conjugate to $(t_0, x_0)$ if $x_1$ is a critical value of $\pi_M:\varphi_{t_1-t_0}^\mathcal{H}\left(T_{x_0}^*M\right) \to M$.
\end{definition}
Hybrid conjugate points (and more generally, hybrid caustics) can be found via the variational equation. Consider the family of Lagrangian manifolds,
\begin{equation*}
	\mathcal{L}_0 = T_{x_0}^*M, \quad \mathcal{L}_t = \varphi_t^\mathcal{H}\left(\mathcal{L}_0\right),
\end{equation*}
where $\varphi_t^\mathcal{H}$ is the flow of the hybrid Hamiltonian system. If $(x(t),p(t))\in \mathcal{L}_t$ such that $\pi_M:\mathcal{L}_t\to M$ fails to have full rank, $x(t)\in\mathcal{C}(\mathcal{L}_t)$, and it can be found in the following way: Consider the variational equation

\begin{equation}\label{eq:hybrid_ve}
	\begin{cases}
		\dfrac{d}{dt}\Phi(t,t_0) = A(t)\cdot\Phi(t,t_0), & (x(t),p(t))\not\in {S}, \\[2ex]
		\Phi(t^+,t_0) = \tilde{\Delta}^{X_H}_*\cdot \Phi(t^-,t_0), & (x(t),p(t))\in {S},
	\end{cases}
\end{equation}
where
\begin{equation*}
	A(t) = \begin{bmatrix}
		\dfrac{\partial^2 H}{\partial x^2} & \dfrac{\partial^2 H}{\partial p\partial x} \\[2ex]
		-\dfrac{\partial^2 H}{\partial x \partial p} & -\dfrac{\partial^2 H}{\partial p^2}
	\end{bmatrix},
\end{equation*}
$\tilde{\Delta}^{X_H}_*:TT^*M|_S\to TT^*M$ is the augmented differential, and $\Phi(t,t_0):T_{(x_0,p_0)}T^*M\to T_{(x_t,p_t)}T^*M$ is the state transition matrix.
Then, the point $x(t)\in \mathcal{C}(\mathcal{L}_t)$ (equivalently, $(x(t),t)$ is conjugate to $(x_0,0)$) when
\begin{equation*}
	\det \Phi_{12}(t,t_0)=0, \quad \text{where} \quad  \Phi(t,t_0) = \begin{bmatrix}
		\Phi_{11}(t,t_0) & \Phi_{12}(t,t_0) \\ \Phi_{21}(t,t_0) & \Phi_{22}(t,t_0)
	\end{bmatrix}.
\end{equation*}
\begin{example}[Bouncing Ball cont.]\label{ex:bball_lagrangian}
	Recall the bouncing ball from Example \ref{ex:bball} with dynamics governed by \eqref{eq:bball_continuous} and \eqref{eq:bball_impact_oscillating}, and suppose that the table is not oscillating so the guard is time-invariant.
	The phase-space, $T^*\mathbb{R}^+\subset\mathbb{R}^2$, consists of the first and fourth quadrants in the plane. 
	Let the initial Lagrangian submanifold be $\mathcal{L}_0 = T^*_{x_0}\mathbb{R}^+$ corresponding to the initial position of the ball $x(0) = x_0$. Conjugate points never occur in the purely continuous dynamics - if $\varphi_t^c$ is the flow of the continuous dynamics, then $\varphi_t^c(T_{x_0}^*\mathbb{R}) \subset \mathbb{R}^2$ is a line. Moreover, the state-transition matrix is
	\begin{equation*}
		\Phi(t,t_0) = \begin{bmatrix}
			1 & \dfrac{1}{m}(t-t_0) \\ 0 & 1
		\end{bmatrix}, \quad \Phi_{12}(t,t_0) = \frac{1}{m}(t-t_0).
	\end{equation*}
	As $\Phi_{12}(t,t_0)\ne 0$ when $t>t_0$, conjugates do not occur. 
	
	When the state undergoes a reset, the state-transition jumps according to $\Phi^+ = J\Phi$, where $J$ is the hybrid Jacobian of \eqref{eq:bball_impact_oscillating} and is given by
	\begin{equation*}
		J = \begin{bmatrix}
			-1 & 0 \\ -\dfrac{2m^2g}{y} & -1
		\end{bmatrix}.
	\end{equation*}
	The resulting values of $\Phi_{12}(t)$ are no longer monotone and a caustic may appear. A qualitative picture of the propagated manifold $\mathcal{L}_t = \varphi_t^\mathcal{H}(T_{x_0}^*\mathbb{R}^+)$ is shown in Figure \ref{fig:bball_fold}.
	A numerical demonstration of the caustic trajectory is shown in Figure \ref{fig:envelope_traj} with the values $x_0 = 1$, $m=1$, and $g=2$. One can view Figure \ref{fig:bball_fold} as a time slice of Figure \ref{fig:envelope_traj}. 
	
	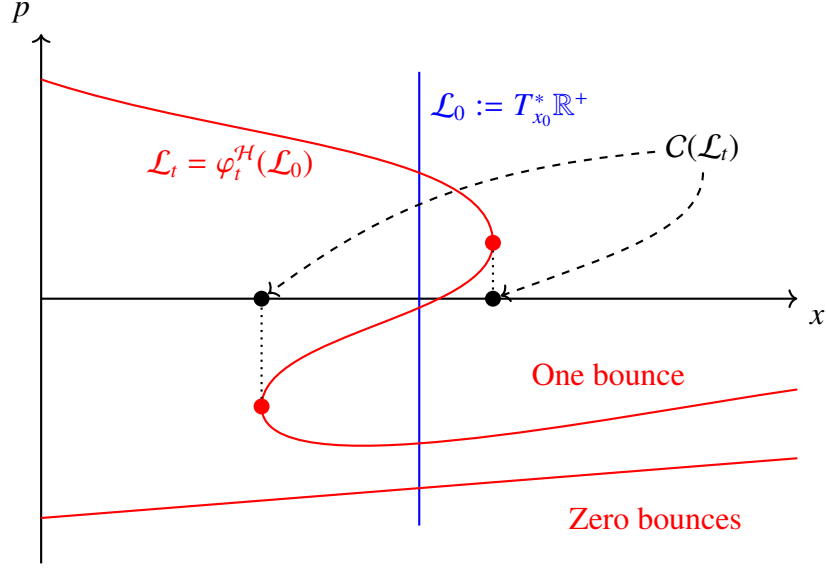
\begin{figure}
		\centering
		\begin{tikzpicture}[xscale=2.5]
			\draw[->,thick,black] (0,0) -- (4,0);
			\draw[->,thick,black] (0,-3.5) -- (0,3.5);
			\node[below right] at (4,0) {$x$};
			\node[above left] at (0,3.5) {$p$};
			\draw[thick, blue] (2,-3) -- (2,3);
			\node[blue, right] at (2,2.5) {$\mathcal{L}_0:=T_{x_0}^*\mathbb{R}^+$};
			\draw[red, thick] (0,-2.9) -- (4,-2.11);
			\draw[red, thick] (0,2.9) to [out=-40,in=90] (2.39,0.739) to [out=-90, in=90] (1.166,-1.425) to [out=-90, in=-160] (4,-1.2);
			\draw[dotted, thick] (2.39,0.739) -- (2.39,0);
			\draw[dotted, thick] (1.166,-1.425) -- (1.166,0);
			\node[dot] (R) at (2.39,0) {};
			\node[dot] (L) at (1.166,0) {};
			\node[dot, red] at (2.39,0.739) {};
			\node[dot, red] at (1.166,-1.425) {};
			\node[red] at (1,1.8) {$\mathcal{L}_t = \varphi_t^\mathcal{H}(\mathcal{L}_0)$};
			\node[red] at (3,-1) {One bounce};
			\node[red] at (3.25,-2.9) {Zero bounces};
			\node (C) at (3.5,2) {$\mathcal{C}(\mathcal{L}_t)$};
			\draw[->, thick, dashed] (C) to [out=190,in=60] (L);
			\draw[->, thick, dashed] (C) to [out=-90,in=40] (R);
		\end{tikzpicture}
		\caption{A plot of the initial and propagated hybrid Lagrangian manifolds for the bouncing ball. The propagated manifold is disconnected where the components depend on the number of resets. Conjugate points occur where $\mathcal{L}_t$ ``folds'' and are indicated by the red points while the caustic is their projection onto the black points.}
		\label{fig:bball_fold}
	\end{figure}
	\begin{figure}
		\centering
		\includegraphics[width=0.95\linewidth]{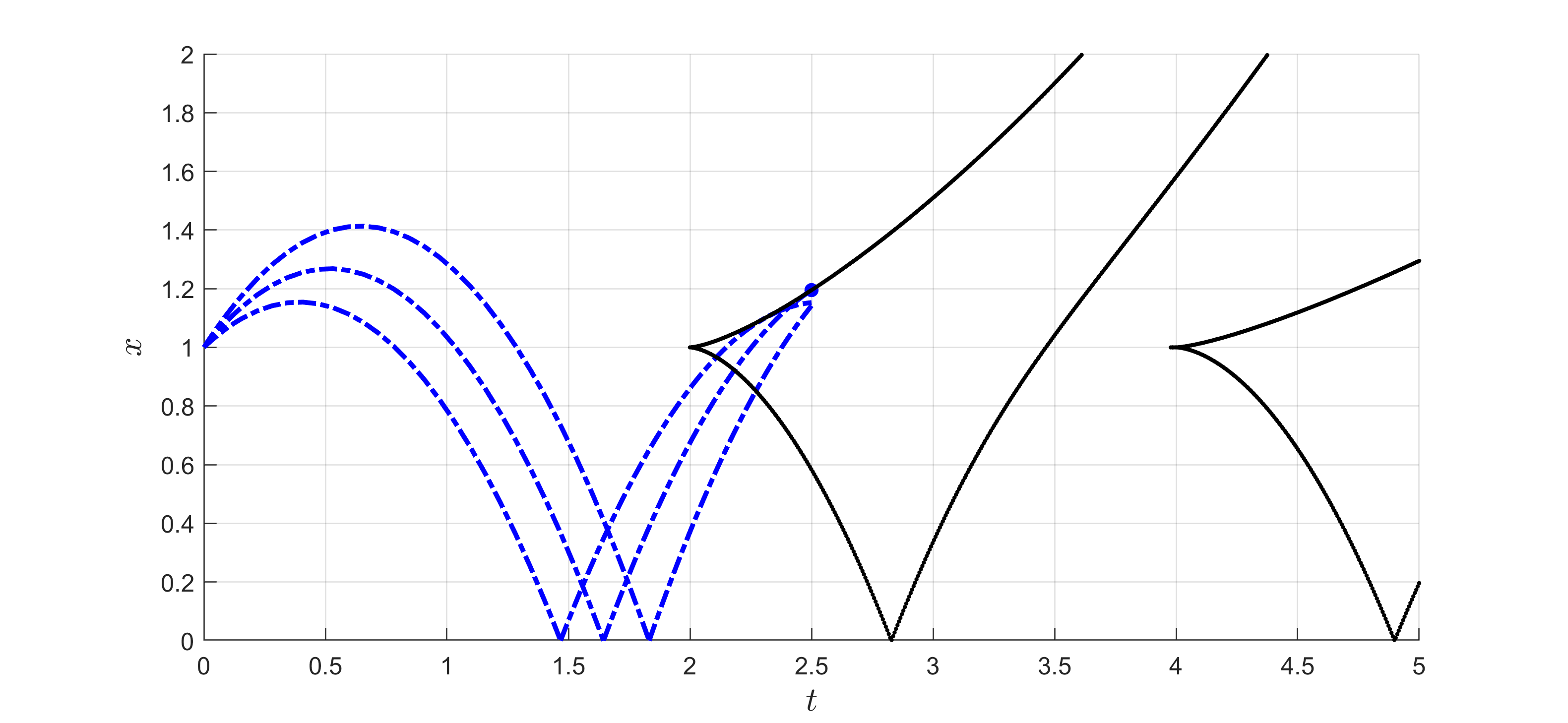}
		\caption{A graph of the caustic trajectory for the bouncing ball. The black curve shows the caustic trajectory from the bouncing ball against time. Conjugate points appear in pairs via the folding shown in Figure \ref{fig:bball_fold}. The blue dashed curves are distinct trajectories that accumulate on the caustic at $t = 2.5$ (to first-order).}
		\label{fig:envelope_traj}
	\end{figure}
	
\end{example}
\section{Irregular executions}\label{sec:Zeno}
The maximum principle above is only valid for regular hybrid arcs. Three phenomena that a non-regular hybrid arc can posses are beating, blocking, and Zeno.
\begin{definition}
	Let $\gamma:[t_0,t_f)\to M$ be a trajectory to the hybrid system $(M,S,X,\Delta)$ with $\mathcal{T} = \{t_k\}\subset [t_0,t_f)$ being the set of reset times. A reset point, $x=\gamma(t_k)$, is $n$-beating if 
	\begin{equation*}
		x,\ \Delta(x),\ \Delta^2(x),\ \ldots \ , \ \Delta^n(x)\in S, \quad \Delta^{n+1}(x)\not\in S,
	\end{equation*}
	and blocking if $\Delta^n(x)\in S$ for all $n\geq 0$. Finally, $\gamma$ is Zeno if the number of impacts $\#\mathcal{T}$ is $ \infty$ and $t_f<\infty$.
\end{definition}
Intuitively, a reset point is beating if there are a finite number of instantaneous impacts happening at the same time, and blocking if the trajectory never leaves the impact surface. In nonlinear systems (especially inelastic impact ones) Zeno can easily happen. While Zeno trajectories can have meaningful extensions past their breaking point \cite{Ames06isthere}, the same cannot be said for blocking states. As such, this section will deal only with the cases of beating and Zeno hybrid arcs.
\subsection{Beating executions}\label{subsec:beating}
Let $\mathcal{HC} = (M,\mathcal{U},\Sigma, f,\Delta)$ be a hybrid control system and let $\overline{\Delta}:\overline{\Sigma}\to M$ be its continuous extension.
Consider the nested sequence of sets:
\begin{equation*}
	\begin{split}
		\Sigma_0 &:= \overline{\Sigma}, \\
		\Sigma_1 &:= \overline{\Sigma}\cap \overline{\Delta}(\Sigma_0), \\
		\Sigma_k &:= \overline{\Sigma}\cap \overline{\Delta}(\Sigma_{k-1}), \\
		\Sigma_\infty &:= \bigcap_{k} \, \Sigma_k,
	\end{split}
\end{equation*}
i.e., $x\in\Sigma_k$ means that (up to closure) $\left\{ x, \Delta(x), \ldots, \Delta^k(x)\right\} \subset S$.
The sets $\Sigma_k$ are called the $k^{th}$-beating sets while $\Sigma_\infty$ is the blocking set. Let us impose the following regularity assumption:
\begin{assume}
	The sets $\Sigma_k$ are all embedded submanifolds (possibly with boundary) and $\Sigma_\infty = \emptyset$.
\end{assume}
Suppose that an optimal trajectory has a 1-beating point at time $t_1\in [t_0,t_f)$, $x(t_1)\in \Sigma_1\setminus \Sigma_2$. The co-state jump equation becomes
\begin{equation*}
	\begin{split}
		p^+ \circ \Delta_*^2 - p^- &\in \mathrm{Ann}(T\Sigma_1), \\
		H^+ - H^- &= 0,
	\end{split}
\end{equation*}
which is under-determined when
\begin{equation*}
	\dim \mathrm{Ann}(T\Sigma_1) = \mathrm{co}\dim \Sigma_1 \geq 2.
\end{equation*}
For concreteness, we will assume that $\mathrm{co}\dim \Sigma_1=2$.
This may be understood as an optimal control problem where the guard is a lower dimensional manifold as is the subject of \cite{pakniyat2015minimum} and can be handled with similar analysis as in \cite{cos_subm}.

Assuming sufficient regularity, denote the 1-dimensional manifold
\begin{equation*}
	\Pi_1 := \left\{ \eta\in\mathrm{Ann}(T\Sigma_1) : H\left( x, p\circ\Delta_*^2 + \eta \right) = H\left( \Delta^2(x), p\right) \right\},
\end{equation*}
i.e., all possible values for $p^-$. Assuming that this occurs at the first reset, the arc must satisfy the following boundary conditions:
\begin{equation}\label{eq:beating_boundary}
	\begin{array}{ll}
		x(t_0) = x_0, & x(t_1)^-\in\Sigma_1, \\
		p(t_0) \in T_{x_0}^*M, & p(t_1)^- \in \Pi_1.
	\end{array}
\end{equation}
These conditions are no longer under-determined and can be solved backward to obtain a unique solution. A similar procedure can be extended to the $k^{th}$-beating set via
\begin{equation*}
	\Pi_k := \left\{ \eta\in\mathrm{Ann}(T\Sigma_k) : H\left( x, p\circ\Delta_*^{k+1} + \eta \right) = H\left( \Delta^{k+1}(x), p\right) \right\}.
\end{equation*}

\subsection{Zeno executions}\label{subsec:Zeno}
There exist two fundamentally different classes of Zeno; one is essentially a disguised version of ``blow up to infinity'' while the other encapsulates ``true Zeno.''
\begin{definition}[Zeno types]
	Let $x\in M$ have a Zeno trajectory with resets occurring at $\{x_k\}_{k=1}^\infty$. The trajectory is spasmodic if the sequence $\{x_k\}_{k=1}^\infty$ escapes every compact set as $k\to\infty$. The trajectory is called steady if it is not spasmodic.
\end{definition}

If the state-space for a hybrid system is compact (or admits compact invariant sets), then the only Zeno trajectories that may occur are of steady type. To assist with steady Zeno, an additional regularity assumption is utilized for the reset map (which is satisfied for mechanical resets).

\begin{definition}[Boundary Identity Property]
	A regular HDS, $\mathcal{H}$, satisfies the boundary identity property if
	$\Delta$ continuously extends to the identity on the boundary of $S$ in $\Delta(S)$, i.e., if $\{s_k\}$ is a sequence in $S$ and $s_k\to s \in \overline{S}\cap\overline{\Delta(S)}$, then $\Delta(s_k)\to s$.
\end{definition}
The boundary identity property is key to extending solutions past the Zeno point in \cite{Ames06isthere}. Although it is important for $\Delta$ to have the boundary identity property, the actual map that requires this is the extended reset, $\tilde{\Delta}$. Additionally, when the boundary identity property holds, the first beating sets coincides with the blocking set, $S_1 = S_\infty$.
\begin{proposition}
	Suppose that the reset map in the regular hybrid control system possesses the boundary identity property. If \eqref{eq:controlled_impact} has a unique solution that extends continuously to $\overline{S}\cap\overline{\tilde{\Delta}(S)}$, then it also possesses the boundary identity property.
\end{proposition}
For a hybrid system whose reset obeys the boundary identity property, (steady) Zeno requires dissipation/contraction. Volume-preservation becomes a convenient obstruction.

\begin{theorem}[\cite{clark2023invariant}]\label{thm:Zeno}
	Suppose that $\mathcal{H}$ is a smooth compact regular hybrid dynamical system whose reset obeys the boundary identity property. 
	Let $\mathcal{N}\subset M$ be the set of all points that have a (steady) Zeno trajectory. 
	If $\mathcal{H}$ preserves a smooth measure, $\mu$, then $\mu(\mathcal{N})=0$, i.e., Zeno almost never occurs.
\end{theorem}
A key reason why the hybrid maximum principle breaks down in the presence of Zeno is due to the degeneracy of the resulting variational/adjoint equations, recall \eqref{eq:hybridVE} and \eqref{eq:hybridAE}. This calculation is demonstrated explicitly for the bouncing ball below.
\begin{example}[Bouncing ball cont.]\label{ex:bouncing_ball_zeno}
	By introducing damping to the reset map \eqref{eq:bball_impact_oscillating}, it is well-known that the trajectories are Zeno. The hybrid maximum principle, Theorem \ref{thm:HPMP}, fails in this setting as state-transition matrix becomes singular which invalidates the hybrid maximum principle \cite{optZeno}. This is explicitly computable in this example; see Proposition \ref{prop:ker_zeno} below.
	
	Let $0<c<1$ and consider the dynamics on $(x,y)\in T^*\mathbb{R}^+$
	\begin{equation}\label{eq:bball_zeno}
		\begin{cases}
			\dot{x} = \dfrac{1}{m}y, \quad \dot{y} = -mg, & x>0, \\
			x\mapsto x, \quad y\mapsto -c^2y, & x=0, \ y<0.
		\end{cases}
	\end{equation}
	A hybrid arc with the initial conditions $x(0)=0$ and $y(0)=y_0>0$ is Zeno with limiting time
	\begin{equation}\label{eq:zeno_time_series}
		t_Z = \frac{2y_0}{mg}\left[ 1 + c^2 + c^4 + \ldots \right] = \frac{2y_0}{mg}\frac{1}{1-c^2}.
	\end{equation}
	Let $z = (x,y)\in T^*\mathbb{R}^+$, then the hybrid variational equation \eqref{eq:hybrid_ve} takes the form
	\begin{equation}\label{eq:bball_variational}
		\begin{cases}
			\delta \dot{z} = A\delta z, & t\not\in\mathcal{T}, \\
			\delta z^+ = J(y(t))\delta z^-, & t\in\mathcal{T},
		\end{cases}
	\end{equation}
	where
	\begin{equation*}
		A = \begin{bmatrix}
			0 & \dfrac{1}{m} \\ 0 & 0 
		\end{bmatrix}, \quad J(y) = \begin{bmatrix}
			-c^2 & 0 \\
			-\dfrac{m^2g}{y}\left( 1+c^2\right) & -c^2
		\end{bmatrix},
	\end{equation*}
	and the event times are the partial sums of the Zeno time,
	\begin{equation*}
		\mathcal{T} = \left\{ \frac{2y_0}{mg}\left( \frac{1-c^{2k}}{1-c^2}\right) : k\geq 0 \right\}.
	\end{equation*}
	If $\Phi(t)$ denotes the state-transition matrix of \eqref{eq:bball_variational}, we have
	\begin{equation*}
		\begin{split}
			\Phi_Z := \lim_{t\nearrow t_Z}, \, \Phi(t) &= \prod_{k=1}^{\substack{\leftarrow \\ \infty}} \,
			J(-c^{2k}y_0)
			\exp\left( A \cdot \left( \frac{2y_0c^{2k}}{mg}\right) \right) 
			= \begin{bmatrix}
				0 & 0 \\
				\dfrac{m^2g}{y_0}\left(\dfrac{1+c^2}{1-c^2}\right) & \dfrac{2}{1-c^2}
			\end{bmatrix},
		\end{split}
	\end{equation*}
	where the product multiplies successive terms on the left. It is clear to see that $\det\Phi_Z=0$, which invalidates the co-state equations required for Theorem \ref{thm:HPMP}.
	\begin{proposition}\label{prop:ker_zeno}
		Let $\zeta:(x_0,y_0)\mapsto t_Z$ be the map that assigns the Zeno time to the initial conditions for the hybrid system \eqref{eq:bball_zeno}. Then, $\ker d\zeta = \ker \Phi_Z$.
	\end{proposition}
	\begin{proof}
		A vector that spans $\ker \Phi_Z$ is 
		\begin{equation*}
			v = \left[ -2y_0, \ m^2g(1+c^2) \right]^\top.
		\end{equation*}
		The Zeno time is 
		\begin{equation*}
			\zeta(x_0,y_0) = \underbrace{\frac{1}{mg}\left(y_0 + \sqrt{y_0^2+2m^2gx_0}\right)}_{\text{time to first impact}} + \underbrace{\frac{2c^2}{mg}\frac{1}{1-c^2}\sqrt{y_0^2+2m^2gx_0}}_{\text{geometric series from \eqref{eq:zeno_time_series}}}.
		\end{equation*}
		Its differential evaluated at $x_0=0$ and $y_0>0$ is
		\begin{equation*}
			d\zeta_{(0,y_0)} = \frac{m(1+c^2)}{y_0(1-c^2)} \ dx_0 +
			\frac{2}{mg(1-c^2)} \ dy_0.
		\end{equation*}
		It follows that $d\zeta_{(0,y_0)}\cdot v = 0$.
	\end{proof}
\end{example}
\section{Application: synchronization for spiking neurons}\label{sec:example}
A system of $n$ neurons can be described using the leaky integrate and fire model \cite{model_neurons}, \cite{neuralbook}:
\begin{equation}\label{eq:neurons_voltage}
	\begin{cases}
		\dot{v}_i =  - v_i + RI_i(t), & v_i < \eta \\
		v_i^+ = v_r, & v_i = \eta
	\end{cases} 
\end{equation} 
where
\begin{itemize}
	\item $v_i$ describes the membrane potential of each neuron, $i=1,\ldots, n$,
	\item $v_r$ is the resting potential of the neuron, which is the charge the neuron would have if it were isolated, 
	\item $I_i = \sum_{j=1}^n \, w_{ij} \alpha_j(t-t^f)+I^{ext}(t) + I_0$ is the current entering each neuron with $I^{ext}$ being the external input current, $w_{ij} < 1$,
	\item $R$ is the membrane conductivity,
	\item $RI_0$ is the natural accumulation of charge inherent to each neuron as a consequence of membrane polarization,
	\item $w_{ij}$ is the strength of the neuron $i$ to neuron $j$ synapse, and
	\item $\alpha_j(t-t^f)$ is a function providing the remaining current from neuron $j$ flowing into neuron $i$ after time $t$, with $\alpha_j(t)=0$ when $t>0$.
	
\end{itemize}
For the sake of simplicity, we will assume that $\alpha\equiv0$, i.e., there is no lag in the synaptic connection and all the current is transmitted during the synapse. Assume, moreover, that the resting potential of each neuron is $v_r=0$ and that the membrane conductivity is $R = 1$. 

Suppose there are $n=2$ neurons with potentials $v_1$ and $v_2$, and decompose the guard via
\begin{equation}\label{eq:decomposed_guard}
	\Sigma = \Sigma^1\cup \Sigma^2 = \left\{(\eta,v_2) : 0\leq v_2 \leq \eta \right\}
	\cup \left\{(v_1,\eta) : 0\leq v_1 \leq \eta \right\}.
\end{equation}
As a discharge is triggered when either neuron reaches its threshold potential, the above decomposition reflects which neuron was fired.  When this happens, the triggering neuron resets its potential to $v_r=0$ and the other receives a sudden increase in its potential proportional to the coupling strength.

Suppose that the control is $u = I^{ext}$, which is only applied to the first neuron. The hybrid control system has the data
\begin{enumerate}
	\item $M = [0,\eta]\times[0,\eta]$,
	\item $\Sigma = \Sigma^1\cup \Sigma^2$ as in \eqref{eq:decomposed_guard},
	\item $\mathcal{U} = \mathbb{R}$ as admissible inputs are unbounded,
	\item $f:M\times\mathcal{U}\to TM$ is given by
	\begin{equation}\label{eq:eom_neurons}
		\begin{split}
			\dot{v}_1 &= -v_1 + I_0 + u, \\
			\dot{v}_2 &= -v_2 + I_0,
		\end{split}
	\end{equation}
	\item $\Delta:\Sigma\to M$ is given by
	\begin{equation}\label{eq:neuron_reset}
		\Delta(v_1,v_2) = \begin{cases}
			(0, w_{12}v_1+v_2), & (v_1,v_2)\in \Sigma^1, \\
			(v_1 + w_{12}v_2, 0), & (v_1,v_2) \in \Sigma^2.
		\end{cases}
	\end{equation}
\end{enumerate}
\begin{remark}
	The system has qualitatively different behaviors depending on whether or not $I_0$ is greater or less than $\eta$, as demonstrated in Figure \ref{fig:neuron_phase}. It will be assumed that $I_0>\eta$.
\end{remark}
\begin{remark}
	There are two issues with the reset map as given by \eqref{eq:neuron_reset}. (a) The first is that $\Sigma^1\cap\Sigma^2 = \{(\eta,\eta)\}\ne\emptyset$ and it does not get mapped to a consistent point under $\Delta$. To resolve this, the point $(\eta,\eta)$ will be taken to be in $\Sigma^1$ and not $\Sigma^2$. This will result in the reset being discontinuous. (b) The second issue is when a reset occurs in the sets
	\begin{equation*}
		\Sigma^1_1 := \{(\eta,v_2) : \eta-w_{12}\eta \leq v_2 \leq \eta\}, \quad
		\Sigma^2_1 := \{(v_1,\eta) : \eta-w_{12}\eta \leq v_1 \leq \eta\},
	\end{equation*}
	as a reset would result in the other neuron having potential greater than $\eta$. As such, a second reset could immediately occur (beating). Addressing both of these issues results in the correct reset map
	\begin{equation}\label{eq:corrected_neuron_reset}
		\Delta(v_1,v_2) = \begin{cases}
			(0, w_{12}\eta + v_2), & v_1 = \eta, \quad \, 0\leq v_2 < \eta-w_{12}\eta, \\
			(w_{12}\eta, 0), & v_1 = \eta, \quad \eta - w_{12}\eta < v_2 \leq \eta, \\
			(v_1+w_{12}\eta, 0), & 0 \leq v_1 < \eta-w_{12}\eta, \quad v_2=\eta, \\
			(0,w_{12}\eta), & \eta-w_{12}\eta \leq v_1 < \eta, \quad v_2 = \eta.
		\end{cases}
	\end{equation}
	Hence, there are 4 different types of impacts (see Figure \ref{fig:neuron_phase} for a schematic drawing):
	\begin{enumerate}
		\item No beating, first neuron is fired: $v_1 = \eta$ and $v_2 < \eta - \eta w_{12}$, i.e., when $(v_1, v_2) \in \Sigma_0^1 := \Sigma^1 -\Sigma_1^1 = \{\eta\} \times [0, \eta(1 - w_{12})]$. The image under the reset map is: $\Delta(\Sigma_0^1) = \{0\}\times [w_{12}\eta, \eta]$.
		\item No beating, second neuron is fired: $v_2 = \eta, v_1 < \eta - \eta w_{12}$, i.e., when $(v_1, v_2) \in \Sigma_0^2 := \Sigma^2 - \Sigma_1^2 = [0, \eta ( 1 - w_{12})]\times \{\eta\}$. The image under the reset map is : $\Delta(\Sigma_0^1) = [w_{12}\eta, \eta]\times \{0\}$.
		\item Beating, first neuron is fired, and an instantaneous spike is induced in the second neuron, i.e., $(v_1, v_2) \in \Sigma_1^1$. Then, image under the reset map is a single point $\{(w_{12}\eta, 0)\}$.
		\item Beating, the second neuron is fired, and an instantaneous spike is induced in the first neuron, i.e., $(v_1, v_2) \in \Sigma_1^2$. Then, image under the reset map is a single point $\{(0, w_{12}\eta)\}$.
	\end{enumerate}
	
\end{remark}
\begin{figure}[H]
	\centering
	\begin{subfigure}{0.45\textwidth}
		\begin{tikzpicture}
			\draw[->, thick] (0,0) -- (6,0);
			\draw[->, thick] (0,0) -- (0,6);
			\node[below right] at (6,0) {$v_1$};
			\node[above left] at (0,6) {$v_2$};
			\foreach \x in {0, 1, 2, 3, 4, 5, 6}{
				\foreach \y in {0, 1, 2, 3, 4, 5, 6}{
					\pgfmathsetmacro{\vx}{-\x + 3}
					\pgfmathsetmacro{\vy}{-\y + 3}
					\draw[->, red] (\x,\y) -- (\x+\vx/5, \y+\vy/5);
				}
			}
			\draw[blue, thick] (4,0) -- (4,6);
			\draw[blue, thick] (0,4) -- (6,4);
			\draw[fill] (3,3) circle [radius=0.1];
		\end{tikzpicture}
		\caption{$I_0<\eta$.}
	\end{subfigure}
	\begin{subfigure}{0.45\textwidth}
		\begin{tikzpicture}
			\draw[->, thick] (0,0) -- (6,0);
			\draw[->, thick] (0,0) -- (0,6);
			\node[below right] at (6,0) {$v_1$};
			\node[above left] at (0,6) {$v_2$};
			\foreach \x in {0, 1, 2, 3, 4, 5, 6}{
				\foreach \y in {0, 1, 2, 3, 4, 5, 6}{
					\pgfmathsetmacro{\vx}{-\x + 5}
					\pgfmathsetmacro{\vy}{-\y + 5}
					\draw[->, red] (\x,\y) -- (\x+\vx/5, \y+\vy/5);
				}
			}
			\draw[blue, thick] (4,0) -- (4,6);
			\draw[blue, thick] (0,4) -- (6,4);
			\draw[fill] (5,5) circle [radius=0.1];
		\end{tikzpicture}
		\caption{$I_0>\eta$.}
	\end{subfigure}
	\caption{Qualitatively distinct cases for the system of two spiking neurons depending on whether $I_0>\eta$ or $I_0<\eta$. }
	\label{fig:neuron_phase}
\end{figure}
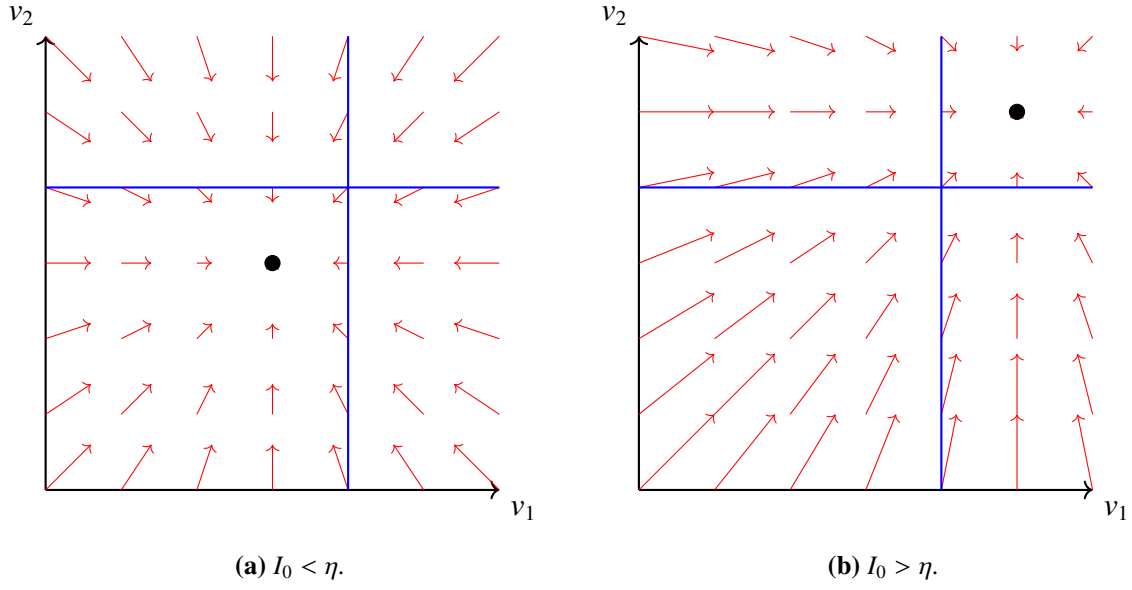
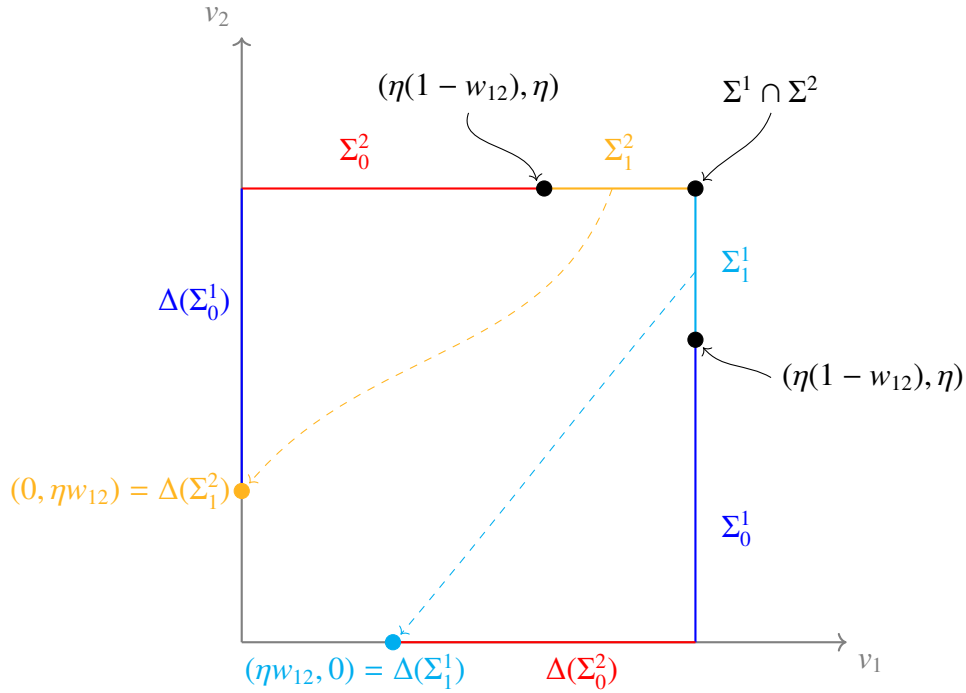
\begin{figure}[H]
	\centering
	\begin{tikzpicture}
		\draw[->, thick, gray] (0,0) -- (8,0);
		\draw[->, thick, gray] (0,0) -- (0,8);
		\node[below right, gray] at (8,0) {$v_1$};
		\node[above left, gray] at (0,8) {$v_2$};
		\draw[red, thick] (0,6) -- (4,6);
		\draw[myorange, thick] (4, 6) -- (6, 6);
		\draw[blue, thick] (6,0) -- (6,4);
		\draw[cyan, thick] (6,4) -- (6,6);
		
		\node[above, black] at (3,7) {$(\eta( 1- w_{12}), \eta) $};
		
		\node[right, blue] at (6.2,1.5) {$\Sigma_0^1 $};
		\node[above, red] at (1.5, 6.1) {$\Sigma_0^2$};
		
		\draw[fill] (6,6) circle [radius=0.1];
		\draw[->] (7,7) to [in=50,out=250] (6.1,6.1);
		\draw[->] (3,7) to [in=100,out=250] (3.9,6.1);
		\node[above] at (7,7) {$\Sigma^1\cap\Sigma^2$};
		\draw[fill, black] (6,4) circle [radius=0.1];
		\draw[fill, black] (4,6) circle [radius=0.1];
		
		\draw[->] (7,3.5) to [in=310,out=150] (6.1,3.9);
		\node[right] at (7, 3.5) {$(\eta(1 - w_{12}), \eta)$};
		
		\node[cyan, right] at (6.2,5) {$\Sigma^1_1$};
		\node[myorange, above] at (5,6.1) {$\Sigma_1^2$};
		
		
		\draw[blue, thick] (0, 2) -- (0, 6);
		\draw[red, thick] (2, 0) --(6, 0);
		
		\draw[fill, myorange] (0,2) circle [radius=0.1];
		\draw[fill, cyan] (2,0) circle [radius=0.1];
		
		\draw[->, dashed,cyan] (6,4.9) to [in=50,out=230] (2.1,0.1);
		\draw[->, dashed, myorange ] (4.9,6) to [in=50,out=250] (0.1,2.1);
		\node[left, myorange] at (0, 2) {$(0, \eta w_{12}) = \Delta(\Sigma_1^2)$};
		\node[below, cyan] at (1.5, 0) {$(\eta w_{12}, 0) = \Delta(\Sigma_1^1)$};
		\node[left, blue] at (0, 4.5) {$\Delta(\Sigma_0^1)$};
		\node[below, red] at (4.5, 0) {$\Delta(\Sigma_0^2)$};
	\end{tikzpicture}
	\caption{Schematic drawing of the different types of resets for the spiking neuron model. Red and blue denote the domain and codomain for no beating resets in neurons 1 and 2, whereas orange and cyan denote the domain and codomain of beating resets in neurons 1 and 2, respectively. Note that both beating sets $\Sigma_1^2 = [\eta(1 - w_{12}), \eta]\times \{\eta\}$ and $\Sigma_1^2 = \{\eta\}\times [\eta(1 - w_{12}), \eta]$ get mapped to a single point $\{(0, \eta w_{12})\} = \Delta(\Sigma_1^2)$ and  $\{(\eta w_{12}, 0)\} = \Delta(\Sigma_1^2)$, respectively.}
	\label{fig:neuron_guards}
\end{figure}

There is extensive evidence in the medical literature \cite{medical1, medical2, medical3} that correlates pathological tremor for patients with neurological disorders such as Parkinson and epilepsy with abnormal synchronization of neuron spiking. This motivates our cost function to be: 
$$J = \int_{t_0}^{t_1} \frac{1}{2} u(t)^2  - 2( v_1 (t)-  v_2(t))^2 dt$$ 
where the goal is to prevent synchronization between the two neurons, at the minimal possible external input. For the sake of simplicity, we measure synchronization by the difference between the potentials of each neuron. Other methods using mutual information, cross correlation and  phase relationships \cite{synchronization} have been proposed in literature, and an in-depth analysis of these methods in the context of control is reserved to future work. 

Hence, the optimal control problem we would like to solve is:
\begin{equation}
	J^*(v_0) := \min_{u(\cdot)} J(u(\cdot),v(t)) \text{ subject to }\begin{cases}
		\eqref{eq:eom_neurons}, \ \eqref{eq:corrected_neuron_reset}, \\
		(v_1(0), v_2(0)) = v_0.
	\end{cases} 
\end{equation}

There are two ways of numerically solving this problem. The first one is using dynamic programming \cite{DP_hybrid}. A numerical implementation of dynamic programming with discretization details given in  Table \ref{tab:neuron_dp_disc} produces the results shown in Figure \ref{fig:neuron_dp}. \begin{table}
	\centering
	\begin{tabular}{c|c}
		Variable & Range \\ \hline
		$v_1$ & \texttt{linspace}(0,1,150) \\
		$v_2$ & \texttt{linspace}(0,1,150) \\
		$u$ & \texttt{linspace}(-2,2,150) \\
		$t$ & \texttt{linspace}(0,1,250)
	\end{tabular}
	\caption{Discretizations for dynamic programming of the 2 neuron model.}
	\label{tab:neuron_dp_disc}
\end{table}
\begin{figure}
	\centering
	\begin{subfigure}{0.45\textwidth}
		\includegraphics[width=\textwidth]{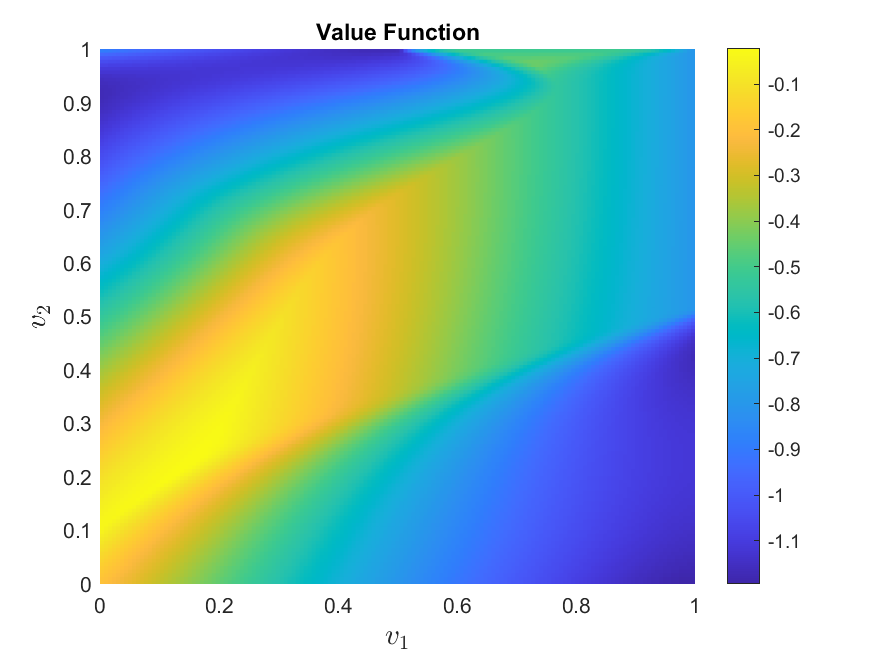}
		\caption{Computed value function \`{a} la DP}.
	\end{subfigure}
	\hfill
	\begin{subfigure}{0.45\textwidth}
		\includegraphics[width=\textwidth]{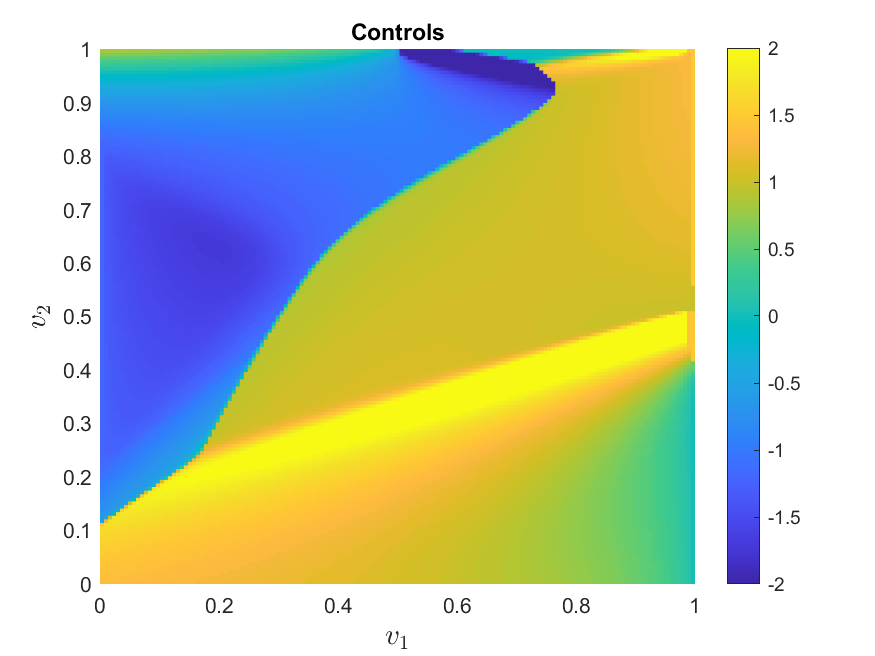}
		\caption{The value of the optimal control, $u^*(t=0; v_1,v_2)$.}
	\end{subfigure}
	\hfill
	\begin{subfigure}{0.45\textwidth}
		\includegraphics[width=\textwidth]{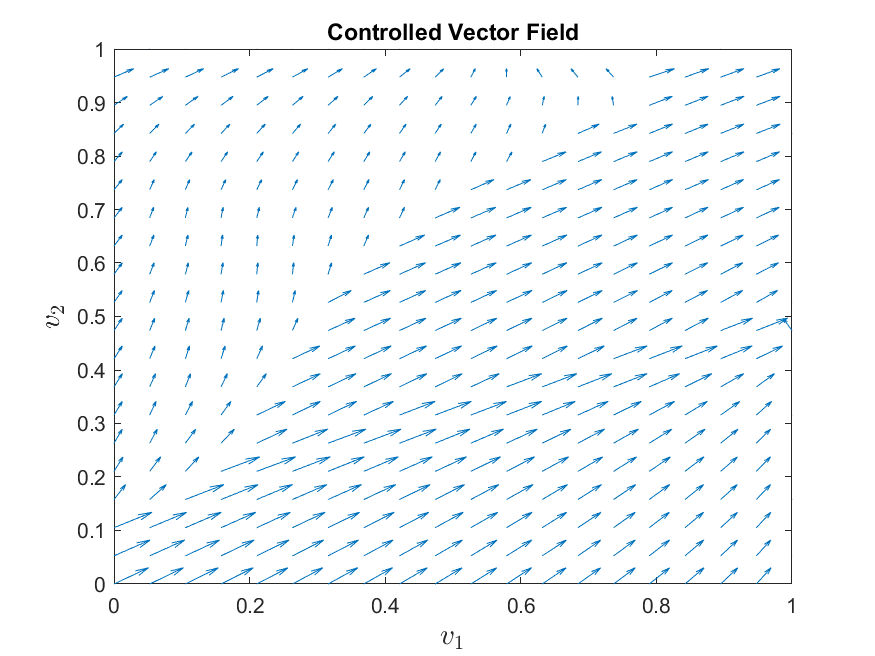}
		\caption{The vector-field with the optimal control at the initial time, $f(v_1,v_2;u^*(t=0, v_1,v_2)$.}
	\end{subfigure}
	\hfill
	\begin{subfigure}{0.45\textwidth}
		\includegraphics[width=\textwidth]{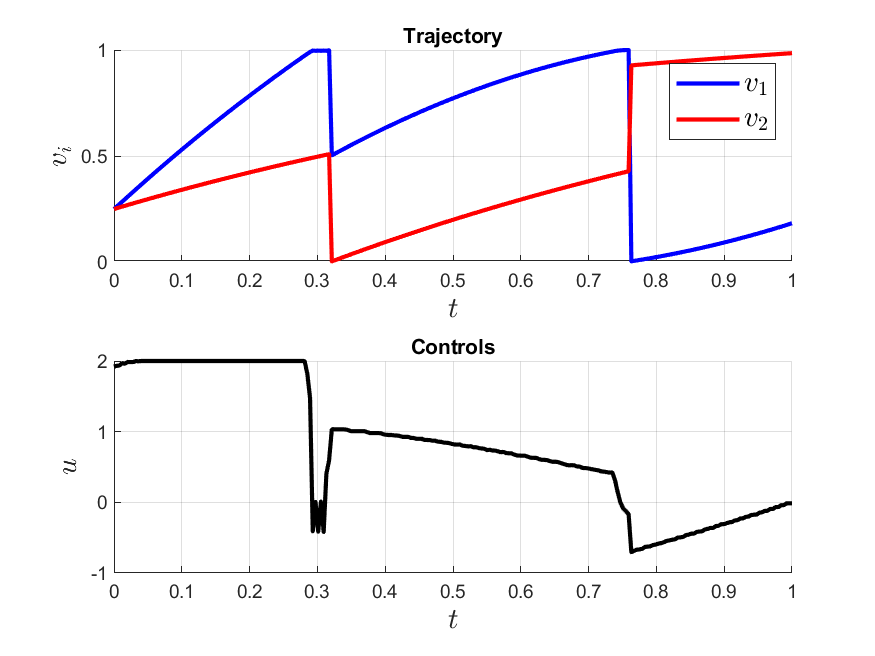}
		\caption{A plot of the trajectories starting at . The two impacts correspond to a beating one and a simple one.}
	\end{subfigure}
	\caption{A summary of the results of dynamic programming applied to the control problem }
	\label{fig:neuron_dp}
\end{figure}

The second method uses the HPMP (Theorem \ref{thm:HPMP}). Let $T$ be the fixed time horizon. Given the optimal Hamiltonian $H:T^*M \times \mathbb{R}\to \mathbb{R}$ and assuming the corner conditions \eqref{eq:corner_conditions} have a unique solution, with the reset $\Delta$ being invertible, we aim to solve the optimal control problem with fixed initial condition $v_0 = (v_1, v_2)_0 $. Since there is no terminal cost, the terminal conditions for the momenta are  $p_v = (p_v^1, p_v^2 ) = (0, 0)$. Assume that the optimality conditions \eqref{eq:minH} allow us to write optimal control strategy $u$ as a function of $v$ and $p_v$, i.e., $u = f(v, p_v)$ for some $f:T^*M \to \mathcal{U}$. Then the optimal solution $v^*, p_v^*$ solves the two point boundary value problem: 
\begin{gather}
	(\dot{v^*} , \dot{p_v^*}) = \left(\frac{\partial H}{\partial p_v}\Big|_{v^*, p_v^*}, - \frac{\partial H}{\partial v}\Big|_{v^*, p_v^*}\right),\quad  (v^*, p_v^*) \notin S \label{eq:cont}\\
	(v^*, p_v^*)^+ = \tilde{\Delta }\big((v^*, p_v^*)^-\big), \quad (v^*, p_v^*) \in S \label{eq:discrete}
\end{gather}
\begin{gather*}
	v^*(0)= v_0\\
	p_v^*(T) = 0
\end{gather*} 
with the optimal control being $f(v^*, p_v^*)$.
We will be looking for solutions that stay within a bounded set $v_{min}^i \leq v_i \leq v_{max}^i$ and $p_{min}^i \leq p_v^i \leq p_{max}^i$ for $i \in \{1, 2\}$. A numerical algorithm for solving this problem is given in Algorithm \ref{algo:numerics}. The strategy is to create a numerical representation of the surface $\mathcal{L}_1 = \{v_0\} \times [p_{min}^1, p_{min}^1] \times [p_{min}^2, p_{max}^2]$  and propagate it forward in time using the dynamics \eqref{eq:cont},\eqref{eq:discrete} to obtain $\varphi^\mathcal{H}_{T}(\mathcal{L}_1)$ and the associated costs at terminal time $T$, $J_T(p^1, p^2)$ for every point in $(p^1, p^2)\in\mathcal{L}_1$. At the terminal time, we compare the cost of the points of $\varphi^\mathcal{H}_T(\mathcal{L}_1)$ that have momenta within an $\epsilon$ precision of 0, and select the ones that have the lowest cost. In the numerical implementation, \texttt{matlab}'s \texttt{ode45} function with event detection was used to solve for the dynamics. The code for the numerical experiments can be found at \url{git@github.com:Maria-A-Oprea/Geometric_hybrid_optimal_control.git}.
\begin{algorithm}
	\caption{Algorithm for solving the hybrid optimal control problem in the case of invertible extended resets. }\label{algo:numerics}
	\begin{algorithmic}
		\State Given $\epsilon$ and number of grid points $N$
		\State Create equidistant Mesh of $N \times N$ points in the rectangle $[p^1_{min}, p^1_{max} ]\times [p^2_{min}, p^2_{max}]$
		\For{$(p^1, p^2)\in$ Mesh}
		\State Compute $p^1_T, p^2_T, J_T(p^1, p^2)$
		\If{ $(p^1_T)^2 +(p^2_T)^2 < \epsilon$ }
		\State Save $(p^1, p^2, J_T(p^1, p^2))$ in \texttt{possible solutions}
		\EndIf
		\EndFor
		\State $p_v^*(0) = $arg$\min_{p^1, p^2 \in \texttt{possible solutions}} J_T(p^1, p^2)$
		\State Compute $v^*(t), p_v^*(t)$ and optimal strategy $u(t) = f(v^*(t), p_v^*(t))$ with initial conditions $v_0, p_v^*(0)$\\
		\Return $v^*(t), p_v^*(t), u(t), J_T(p_v^*(0)) $
	\end{algorithmic}
\end{algorithm}

For our specific case,  the extended Hamiltonian with $p_0 = 1$ is 
\begin{equation*}
	\hat{H}(v_1,v_2,u;p_1,p_2) = \frac{1}{2}u^2 - 2\left(v_1-v_2\right)^2 + p_1\left(-v_1 + I_0 + u\right) + p_2\left(-v_2+I_0\right),
\end{equation*}
and the optimality conditions become:
$$\frac{\partial \hat{H}}{\partial u} = 0 \implies u^* = - p_1$$
Plugging this back into the extended Hamiltonian, we obtain the optimal Hamiltonian:
$$  H^*(v_1, v_2, p_1, p_2) = - \frac{1}{2}p_1^2 -2(v_1 - v_2)^2 + p_2(-v_2 + I_0) + p_1(-v_1 + I_0)$$

Hence, the equations of motion for the continuous part of the extended system are:
\begin{equation}\label{eq:eom_hpmp_neurons}
	\begin{cases}
		\dot{v}_1 = - p_1 - v_1 + I_0 \\
		\dot{v}_2 = -v_2 + I_0 \\
		\dot{p}_1 = p_1 + 4(v_1 -v_2)\\
		\dot{p}_2 = p_2 + 4(v_2 - v_1)
	\end{cases}
\end{equation}

There are four distinct corner conditions, determined by which of the four regions—$\Sigma^1_0$, $\Sigma^1_1$, $\Sigma^2_0$, or $\Sigma^2_1$—the trajectory intersects. 
\begin{align*}
	&\Sigma_0^1:\begin{cases}
		p_1^- = p_1^+ - \varepsilon, \\
		p_2^- = p_2^+,
	\end{cases} \hspace{ 2cm}
	\Sigma_0^2:\begin{cases}
		p_1^- = p_1^+, \\
		p_2^- = p_2^+ - \varepsilon,
	\end{cases}\\
	&\Sigma_1^1:\begin{cases}
		p_1^- = \varepsilon\\
		p_2^- = 0,
	\end{cases} \hspace{3cm}
	\Sigma_1^2:\begin{cases}
		p_1^- = 0\\
		p_2^- = \varepsilon,
	\end{cases}
\end{align*}
where $\varepsilon$ is chosen for each case such that the Hamiltonian is conserved, i.e., $H^*(v_1^-, v_2^-, p_1^-, p_2^-) = H^*(v_1^+ , v_2^+, p_1^+, p_2^+)$. Solving for $\varepsilon$ results in a linear equation for $\Sigma_2^0$, $\Sigma_2^1$ and a quadratic equation for $\Sigma_0^1$ and $\Sigma_1^1$.

Note that on the beating sets, the corner conditions are derived using Equation \eqref{eq:beating_boundary}. Since $\Delta|_{\Sigma_1^1}$ is a constant, $\Delta_*|_{\Sigma_1^1} = 0$, and the corner conditions become $[p_1^+ \ p_2^+] \begin{bmatrix}
	0 & 0\\0 & 0
\end{bmatrix} = [p_1^- \ p_2^-] + [\epsilon \ 0]$ such that $H(\eta w_{12}, \eta, p_1^+, p_2^+) = H(\eta, v_2^-, \varepsilon, 0)$. An analogous argument is used to derive the corner conditions on $\Sigma_1^2$. The condition that $p_1 = 0$ on $\Sigma_1^2 = [0.5, 1]\times \{1\}$ is also verified by the numerical computation of the optimal control $u^* = -p_1$ using dynamic programming (see Figure \ref{fig:neuron_dp}b)

The costate equations can be solved backward. Since there is no terminal cost $g = 0$,  the terminal conditions for the costate are $p_1(T) = p_2(T) = 0$. Thus, finding the optimal trajectory from initial condition $v_0$ amounts to solving the two point boundary value problem given by \eqref{eq:eom_hpmp_neurons} with impact surface $S = \{(v_1, v_2, p_1, p_2)| (v_1, v_2) \in \Sigma^1 \cup \Sigma^2\}$, reset conditions given by \eqref{eq:corrected_neuron_reset} and the corner conditions above, and boundary conditions $(v_1(0) , v_2(0)) = v_0$, $p_1(T) = p_2(T) = 0$. Figure \ref{fig:edp_pmp_comparison_neuron} shows a comparison of the optimal trajectory coming from dynamic programming, and the solution to the hybrid PMP obtained with a shooting method. 

Figure \ref{fig:neuron_manifolds} depicts the terminal manifold $\tilde{\mathcal{L}} = \{(v_1, v_2, 0, 0)| v_1, v_2 \in [0, 1]\}$ propagated backward until the first impact happens, denoted by $ \varphi^\mathcal{H}_{-t} (\tilde{\mathcal{L}}):=\tilde{\mathcal{L}}_{-t}$, and the initial manifold $T_{(0, 0)}^*[0, 1]^2$ propagated forward until the first impact happens, denoted by $\varphi^\mathcal{H}_{t} T_{(0, 0)}^*[0, 1]^2:= \mathcal{L}_t$. 
The impact region is divided into three subregions: $I = \pi_{v_1, p_1} (\pi^*)^{-1}(\Sigma_1^1)$, $II = \pi_{v_1, p_1} \pi^{-1}(\Sigma_2^0)$, and $III = \pi_{v_1, p_1} \pi^{-1}(\Sigma_2^1)$, where $\pi_{v_1, p_1}(v_1, v_2, p_1, p_2) = (v_1, p_1)$ and $\pi:T^*[0, \eta]^2 \to [0, \eta]^2$ is the cotangent bundle projection. In other words, we are showing the corresponding $p_1$ fibers, at locations where an impact happens either in $v_1$ or $v_2$. Figure \ref{fig:neuron_manifolds} shows the image of each region $I, II$, and $III$ under the reset map. The corner conditions give us the value of the jump in $p_1$ between the intersection of the regions $I, II, II$ with $\mathcal{L}_t$ (blue) and the intersection of the corresponding points in regions $\Delta(I), \Delta(II), \Delta(III)$ with $ \tilde{\mathcal{L}}_{-t}$ (green). 

\begin{figure}[H]
	\centering
	\begin{subfigure}{0.45\textwidth}
		\includegraphics[width=\textwidth]{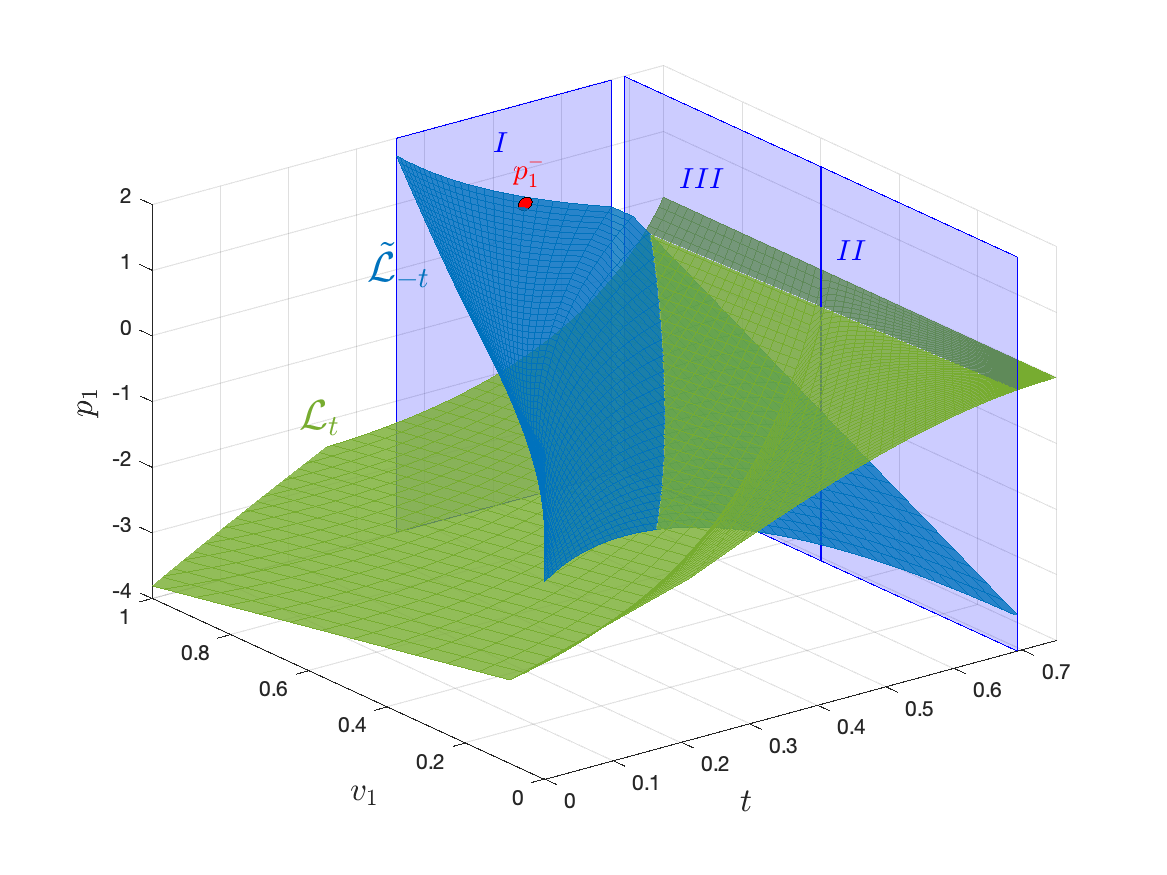}
		\caption{The impact surface is depicted in blue. }
	\end{subfigure}
	\hfill
	\begin{subfigure}{0.45\textwidth}
		\includegraphics[width=\textwidth]{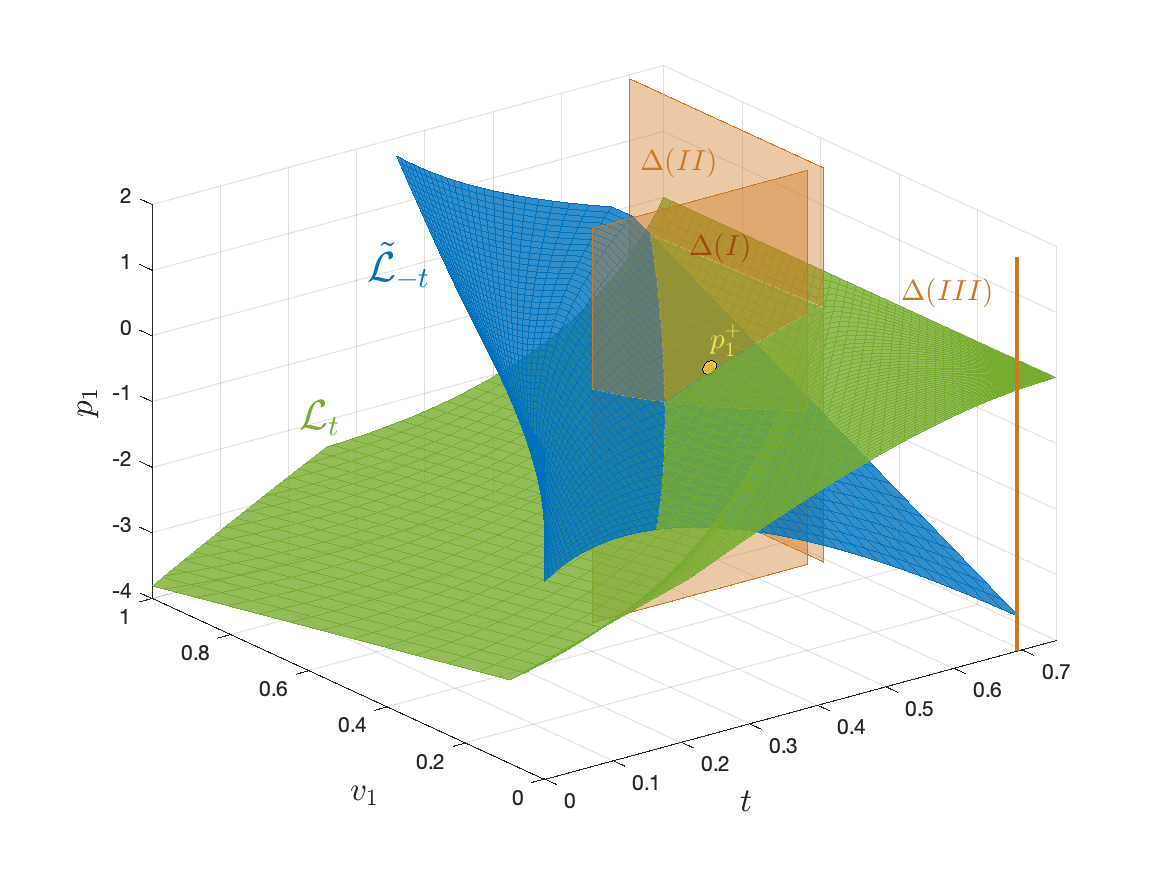}
		\caption{The image of the impact surface is depicted in brown.}
	\end{subfigure}
	\hfill
	\caption{A figure of the Lagrangian submanifolds $\varphi_{t}^\mathcal{H}(T^*_{(0, 0)} M)$ and $\varphi_{-t'}^\mathcal{H}(\Lambda_T)$ for times $t$ smaller than the first forward impact time and $t'$ larger than the first backward impact time. The solution to the corner condition at $p_1^+$ represented by the yellow point in (b) is the red point $p_1^-$ in (a). }
	\label{fig:neuron_manifolds}
\end{figure}

\begin{figure}[H]
	\centering
	\includegraphics[width=0.5\linewidth]{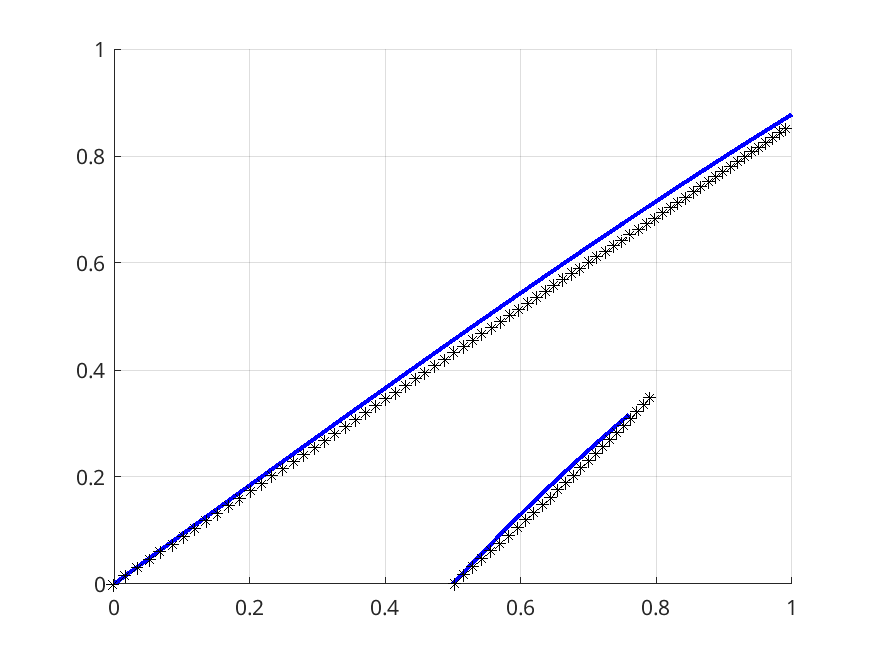}
	\caption{Comparison of the trajectories obtained with dynamic programming (black) and the hybrid maximum principle (blue).}
	\label{fig:edp_pmp_comparison_neuron}
\end{figure}

\section{Conclusions}
\label{sec:conclusions}

This work studied the geometric interpretation of optimal control problems in the context of deterministic hybrid dynamical systems. For regular optimal arcs, the co-states jump according to the ``Hamiltonian jump condition'' \eqref{eq:controlled_impact}, which preserves the canonical symplectic structure. As the symplectic structure is preserved, the notion of Lagrangian submanifolds, caustics, and conjugate points can be interpreted in the hybrid setting. 
For irregular optimal arcs, the maximum principle remains applicable to beating arcs but fails for Zeno. Although Zeno can be ruled out almost everywhere for ``reasonable'' systems (compact, smooth, regular, and boundary identity property), optimal trajectories may still be Zeno which is outside the scope of this current work. We point out two areas of future work.
\subsection{Hybrid conjugate points}
The maximum principle only presents necessary conditions for a local extremum. Sufficient conditions can be achieved by the exclusion of conjugate points along the arc, cf. Theorem 21.3 in \cite{ascontrol}. This observation has less utility in the hybrid domain as shown in Figure \ref{fig:bball_fold}; the Lagrangian manifold can be in multiple pieces as a result of resets, which can result in multiple points lying in $T^*_{x_0}M\cap\varphi_{-t_f}^\mathcal{H}(\mathcal{L})$, even when the caustic trajectory is empty. A potential plan to study these solutions is by examining the geometry of the intersection \eqref{eq:Lagrangian_intersection}.

Another question is to see how the notion of caustics in Section \ref{subsec:caustics} matches with caustics in the billiard community, \cite{gutkin}.

\subsection{Numerical techniques for irregular arcs}
Numerically implementing the boundary value problem in Theorem \ref{thm:HPMP} for regular arcs is difficult as, e.g., the shooting map is highly discontinuous \cite{optZeno}. This difficulty is compounded for irregular arcs, either beating or Zeno. At the onset of the problem, neither the number of resets nor their location is known. If beating may occur, it is additionally unknown in which beating set the resets happens. This adds difficulty as each case is a separate boundary value problem \eqref{eq:beating_boundary}.  
Moreover, (especially for beating) the state trajectory is only well-defined going forward while the co-state trajectory is only well-defined going backward. It would be desirable to be able to construct an algorithm that can determine whether or not beating transpires apriori and be able to decouple the forward and backward dynamics. Another approach could be to adapt pseudo spectral methods \cite{pseudo_spectral_book} to these problems.

Zeno arcs are difficult due to their singular nature, e.g., Example \ref{ex:bouncing_ball_zeno}. A contributing factor to their difficulty is that extension of solutions presented in \cite{Ames06isthere} are, in a sense, ``weak'' solutions to the HDS. A possible approach is to find a suitable notion of a hybrid viscosity solution to the resulting HJP problem where the singularity is blurred away.

\section*{Author contributions}
\textbf{William Clark}: Conceptualization, Formal Analysis, Writing - original draft, Writing - review \& editing, Supervision;
\textbf{Maria Oprea}: Methodology, Investigation, Software, Writing - review \& editing.

\section*{Use of AI tools declaration}
The authors declare they have not used Artificial Intelligence (AI) tools in the creation of this article.

\section*{Acknowledgments}

This work was funded by the NSF grant DMS-1645643 and AFOSR Award No. MURI FA9550-32-1-0400.

\section*{Conflict of interest}

The authors declare no conﬂict of interest.

\color{black}

\end{document}